\documentclass{siamart251104}
\usepackage{makeidx}        
\usepackage{graphicx}       
\usepackage{multicol}       
\usepackage[bottom]{footmisc}

\usepackage{newtxtext}       %
\usepackage[varvw]{newtxmath}       

\usepackage{amsmath}
\usepackage{float}
\usepackage{diagbox}
\usepackage{bbold}
\usepackage{hyperref}
\usepackage{enumitem}
\usepackage{graphicx} 
\usepackage{caption} 
\usepackage{subcaption} 
\usepackage{rotating} 

\newtheorem{thm}{Theorem}

\newtheorem{lem}{Lemma}

\newtheorem{rmk}{Remark}

\makeindex

\newcommand{\dx}{\Delta x}
\newcommand{\dt}{\Delta t}

\newcommand{\mI}{\m{I}}
\newcommand{\mA}{\m{A}}
\newcommand{\mB}{\m{B}}

\newcommand{\mAh}{\m{\hat{A}}}

\newcommand{\cU}{{\mathcal{U}}}

\newcommand{\Aep}{\mA^{\epsilon} }
\newcommand{\bep}{b^{\epsilon} }
\newcommand{\cep}{\vc^{\epsilon}}

\newcommand{\At}{\tilde{\mA} }

\newcommand{\m}[1]{\mathbf{#1}}

\newcommand{\vy}{\m{y}}
\newcommand{\vz}{\m{z}}
\newcommand{\vb}{\m{b}}

\newcommand{\ve}{\m{e}}
\newcommand{\vE}{\m{E}}
\newcommand{\vc}{\m{c}}

\title{Smooth perturbations of diagonally implicit Runge--Kutta methods}

\author{John Driscoll\thanks{Mathematics Department, University of Massachusetts Dartmouth,
285 Old Westport Road, North Dartmouth, MA 02747.}\and Sigal Gottlieb*\thanks{sgottlieb@umassd.edu}\and
Zachary J. Grant*\and
César Herrera\thanks{Department of Mathematics, Purdue University, 150 North University Street.
West Lafayette, Indiana 47907}\and
Tej Sai Kakumanu*\and
Monica Stephens\thanks{Mathematics Department, Spelman College, 350 Spelman Lane S.W.  Atlanta, GA 30314} 
}

\begin{document}

\maketitle

\bibliographystyle{siam}

\noindent{\bf Abstract:}
{\em A mixed accuracy framework for Runge–Kutta methods  
presented in \cite{Grant2022} has been shown to speed up the computation 
in diagonally implicit Runge–Kutta (DIRK) methods by using
less expensive low accuracy approaches for the implicit stages.
This theory included both smooth and nonsmooth perturbations,
and subsequent work focused primarily on the case of nonsmooth 
perturbations that arise from mixed precision simulations.
In this work the focus is on  smooth perturbations that arise  
from  using  less accurate models or under-resolved iterative solvers
to simplify the implicit computations. 
We develop an accuracy and stability analysis 
based on the framework in \cite{Grant2022} to design 
methods that strategically replace the original operator by a lower accuracy operator to
reduce computational cost while mitigating the effect of the perturbations. 
In particular, we focus on designing novel methods that are high order for 
smooth  perturbations  that satisfy additional local consistency conditions. 
Finally, we verify the performance of  the novel perturbed DIRK methods designed in this work
and  numerically study the impact of different  types of smooth perturbations 
on the accuracy and stability of the methods.
}

\section{Overview}

We begin with an evolutionary partial differential equation (PDE)
\begin{eqnarray} \label{PDE}
    \cU_t  &=& F(\cU,\cU_x, \cU_{xx}), \; \; \; \cU(0,x) = \cU_0(x),
\end{eqnarray} 
and discretize it in space,  resulting in 
an autonomous  ordinary differential equation (ODE)
\begin{eqnarray} \label{ODE}
    u_t &=& f(u), \; \; \; y(0) = y_0.
\end{eqnarray} 
This ODE can then be  evolved  forward using a standard  Runge--Kutta  method,
e.g. a diagonally  implicit  Runge--Kutta  (DIRK) method of the form:
\begin{eqnarray} \label{DIRK}
    u^{(i)} & =  u_n + \dt \sum_{j=1}^{i} a_{ij} f(u^{(j)}),   \; \; \; \; 
    u_{n+1} & =  u_n + \dt \sum_{i=1}^s b_{i} f(u^{(i)}) .
\end{eqnarray}
The  coefficients $a_{ij}$ and $b_i$ are written as a matrix $\mA = (a_{ij})$, 
and a vector $\vb = (b_i)$, and must  satisfy the conditions
\begin{subequations} \label{Coefficients}
\begin{eqnarray} 
    a_{ii} & \geq & 0, \; \; \; b_i  >  0 ,\; \; \; 
    c_i =   \sum_j a_{ij} \; \; \; \mbox{are distinct} \\
    M & = &  \mB \mA + \mA^T \mB - \vb \vb^T 
    \; \; \; \mbox{is SPD} .
\end{eqnarray} 
\end{subequations}
If $f \in C^1(a,b)$  is a contractive function so that 
$ \left( z -  y , f(z) -  f(y) \right)  \leq 0, \; \; \; \forall \;  y,z,$
the conditions \eqref{Coefficients} ensure that the method \eqref{DIRK}  satisfies an
inner-product  stability property known as B-stability \cite{CrouzeixBstab}.

For computational efficiency, we can replace $f(u^{(i)})$ 
with  $f_{\varepsilon}(u^{(i)})$  that is less expensive to invert:
\begin{align}  \label{MADIRK}
u^{(i)}&= u_n + \dt \sum_{j=1}^{i} \tilde{a}_{ij} f(u^{(j)}) 
+   \dt \sum_{j=1}^{i} a^\varepsilon_{ij} f_\varepsilon(u^{(j)}) \nonumber \\
u_{n+1}&= u_n + \dt \sum_{i=1}^{s} {b}_{i} f(u^{(i)})    
\end{align}
Here we let $\tilde{a}_{ij} = a_{ij} - a^\varepsilon_{ij}$.
Correspondingly, we define the matrix $\At = \mA - \Aep$.

A  mixed accuracy framework for Runge--Kutta methods was presented in \cite{Grant2022}
for both the smooth and non-smooth perturbation $f_\varepsilon$.
In prior work \cite{Grant2022, Burnett1, Burnett2, Driscoll2026}
this approach was applied to design diagonally perturbed
DIRK methods of the form \eqref{MADIRK}  where $\Aep$ is a diagonal matrix. 
In this case, the operator $f_\varepsilon$ was a low accuracy or low precision version of $f$,
which resulted in a less expensive implicit solve at each stage.
In \cite{Driscoll2026}, we focused on understanding the effect of 
non-smooth perturbation errors on the stability of the diagonally perturbed DIRK methods, 
using the  mixed-accuracy approach in \cite{Grant2022}. 
We also designed stabilized  correction approaches that allowed us to improve 
the accuracy and stability of the perturbed DIRK methods in the case of 
smooth or non-smooth perturbations.

In this work we take advantage of the flexibility afforded when the perturbation 
results from such scenarios as linearization, less accurate spatial discretizations, 
and under-resolved iterative solvers, and so is smooth.
We show that when $f_\varepsilon$ is a {\bf smooth} perturbation of $f$, 
and in addition the error from this perturbation satisfies certain consistency 
conditions, then the order conditions simplify considerably. 
In this work we investigate methods of the form \eqref{MADIRK} where $\Aep$ is 
not required to be strictly diagonal. 
We use  the perturbed B-series framework presented in \cite{Grant2022}
as well as additional simplifying conditions on the perturbation error
to develop appropriate order conditions (Section \ref{sec:MADIRK})
and design novel perturbed DIRK methods  that are stable and highly accurate
(Section \ref{sec:NovelMethods}). 
Our novel perturbed DIRK methods retain high accuracy solutions by canceling 
out the perturbations that arise  from replacing $f$ with $f_\varepsilon$.
In Section  \ref{sec:NumericalTests}  we study these new methods and compare 
them to the diagonally perturbed methods on a variety  of test cases.
Although the framework and novel methods in this work are designed for 
{\em any} smooth perturbations that satisfy
some additional consistency conditions, 
in our numerical tests we focus specifically on  linearizations of $f$.  
We show that the qualities of each linearization as well as the 
value that we  linearize around will strongly impact the accuracy of the 
overall method. 



\section{Mixed accuracy DIRK methods for smooth perturbations} \label{sec:MADIRK}


The additive DIRK scheme \eqref{MADIRK} evolves the PDE \eqref{PDE} forward using both the 
operator $f$ and $f_\varepsilon$. For analysis, it is convenient to introduce the
perturbation error $\tau = f_\varepsilon  - f$,  so that \eqref{MADIRK} can be re-written in
the perturbed form
\begin{align} \label{eq:PerturbedMethod}
u^{(i)}&= u_n + \dt \sum_{j=1}^{i} a_{ij} f(u^{(j)}) +   
\dt \sum_{j=1}^{i} a^\varepsilon_{ij} \tau(u^{(j)}) \nonumber \\
u_{n+1}&= u_n + \dt \sum_{i=1}^{s} b_{i} f(u^{(i)}) .
\end{align}

In \cite{Grant2022} the order conditions were derived for smooth 
as well as non-smooth perturbations.
In this work we consider only smooth perturbations, and additionally 
simplify the order conditions in  \cite{Grant2022} by  requiring 
that the perturbed function has the property
\begin{eqnarray} \label{eq:Tau}
    f_\varepsilon(u_n) = f(u_n) \; \; \Rightarrow \; \; \tau(u_n) = f_\varepsilon(u_n) - f(u_n) =0 .
\end{eqnarray}
This property is natural in the context of a linearization, where we can expect that 
the linearization of  a function around the point $u_n$ recovers the correct value at $u_n$. 
Note that if we do not linearize around $u_n$ we will lose this property.

Observe that in the form \eqref{eq:PerturbedMethod} we avoid using 
the perturbed operator $f_\varepsilon$ 
in the final reconstruction stage. If we did not do so, we would 
have coefficients $\bep$ (similar to the $\Aep$) in the final stage,
but we set these to zero:
\begin{eqnarray} \label{eq:bep}
\bep_i = 0 , \; \; \;  \forall \; \; i.
\end{eqnarray}
The smoothness of the perturbation coupled with the local consistency
condition \eqref{eq:Tau} and the property \eqref{eq:bep}  significantly 
simplify the order conditions, given in Table \ref{tab:OC} below.
\begin{table}[H]  
\centering
\begin{tabular}{|lllll|} \hline
type & error & expansion       & order  & note \\ 
& term &  coefficient      &  condition &  \\ \hline
M & $\dt$    & $f$             & $\vb \ve = 1$                           & condition on the method \\
M & $\dt^2$  & $f_u f$         & $\vb \vc = \frac{1}{2}$                 & condition on the method  \\
P & $\dt^2$  & $f_u \tau $     & $\vb \cep = 0 $                         & vanishes if $\tau(u_n) = 0$\\ 
\hline
M & $\dt^3$  & $f_u f_u f $    & $\vb \mA \vc = \frac{1}{6}$             & condition on the method \\ 
M & $\dt^3$  & $f_{uu} (f,f) $ & $\vb (\vc \cdot \vc) =\frac{1}{3} $     & condition on the method \\ 
P & $\dt^3$  & $f_u \tau_u f $ & $\vb \Aep \vc = 0 $                     & vanishes if $\tau_u(u_n) = 0$ \\ \hline 
M & $\dt^4$  & $f_u f_u f_u f $   & $\vb \mA  \mA \vc = \frac{1}{24}$       & condition on the method \\ 
M & $\dt^4$  & $f_u f_{uu} (f,f)$ & $\vb \mA (\vc \cdot \vc) = \frac{1}{12}$ & condition on the method  \\
M & $\dt^4$  & $f_{uu} (f_u f,f)$ & $\vb (\mA \vc \cdot \vc) = \frac{1}{8}$ & condition on the method  \\
M & $\dt^4$  & $f_{uuu}(f,f,f)$   & $\vb (\vc \cdot \vc \cdot \vc) = \frac{1}{4}$ & condition on the method  \\
P &  $\dt^4$  & $f_u f_u \tau_u f$ & $\vb \mA \Aep \vc = 0$ & vanishes if $\tau_u(u_n) = 0$  \\
P & $\dt^4$  & $f_u \tau_u f_u f$ & $\vb \Aep  \mA \vc = 0$ & vanishes if $\tau_u(u_n) = 0$  \\
P & $\dt^4$  & $f_u \tau_u \tau_u f$ & $\vb \Aep \Aep \vc = 0$ & vanishes if $\tau_u(u_n) = 0$  \\
P & $\dt^4$  & $f_{uu} (f,\tau_u f)$ & $\vb (\vc \cdot \Aep \vc) = 0$ & vanishes if $\tau_u(u_n) = 0$ \\
P & $\dt^4$  & $f_u \tau_{uu} (f,f)$ & $\vb \Aep (\vc \cdot \vc) = 0$ & perturbation condition \\ \hline
\end{tabular}
\caption{Order conditions of the perturbed DIRK method \ref{eq:PerturbedMethod}, up to fourth order. 
These conditions assume that conditions \eqref{eq:bep} and \eqref{eq:Tau} are both satisfied.
\label{tab:OC}}
\end{table} 

\smallskip

In Table \ref{tab:OC} we observe the  two types of conditions: 
the first type  (marked with "M" in the left column) are 
the  conditions on the method coefficients $\mA, \vb$, which multiply terms coming only from $f$, 
and that need to be satisfied for the method to give the correct 
accuracy in the absence of the perturbation.
The second type of conditions (marked "P" in the left column) 
are those that reduce the impact of the perturbation.
These come from the terms involving the derivatives of $\tau$ at $u_n$.
Recall that we require that $\tau(u_n)=0$, which reduces the number of such conditions
(we list in the table only the first one). If we additionally require that the derivative 
satisfies a  local consistency condition
\begin{eqnarray} \label{eq:TauPrime}
\tau_u(u_n) = 0 
\end{eqnarray} 
then we eliminate many more perturbation conditions.

In Table \ref{tab:OC}  we see that if the perturbation  satisfies the condition \eqref{eq:Tau}  
(e.g. we linearize around the  point $ u_n$) then methods of the form \eqref{MADIRK} 
give us a second order  as long as the method's order conditions are satisfied, 
and we do not require  any additional conditions on the perturbation coefficients. 
If the method's coefficients $\mA, \vb $ satisfy the third order conditions 
and the perturbation satisfies  \eqref{eq:Tau}, then one additional perturbation condition
$\vb \Aep \vc = 0 $  will allow third order convergence. 
Notably, this condition  too vanishes if the perturbation is chosen so 
that \eqref{eq:TauPrime} is also satisfied.
Finally, if a method's  coefficients $\mA, \vb $ satisfy the fourth order conditions
and the perturbation satisfies \eqref{eq:Tau} and \eqref{eq:TauPrime}, the 
we are able to get {\em fourth order}  convergence with only one additional condition: 
$\vb \Aep (\vc \cdot \vc) = 0$.
However, if the perturbation does not satisfy $\tau_u(u_n) \neq 0$ then we require 
$\vb \Aep \vc = 0 $ for third order, and  five additional conditions for fourth order.

The local consistency conditions \eqref{eq:Tau} and \eqref{eq:TauPrime} can be satisfied 
by many smooth perturbations. They are particularly natural in the setting of linearization.
In the table below we show some examples of linearizations around $u_n$, 
(note that we define $U_n$  as the  matrix with diagonal $u_n$, and that 
). 
we  observe that  linearization (1) satisfy the first local consistency condition
\eqref{eq:Tau}, but not the second local consistency condition \eqref{eq:TauPrime}.
The linearizations in (2) may not satisfy either local consistency condition, depending on the 
behavior of the differentiation matrix. The Taylor series linearization (3) will satisfy
both local consistency condition \eqref{eq:Tau} and \eqref{eq:TauPrime}.
A discussion is in Appendix A.2. \\

{\small
\begin{tabular}{|l|l|l|} \hline
 & $\cU_t =  (\cU^m)_x$ & $\cU_t =  (\cU^m)_{xx}$ \\ \hline
 1 & $D_x(U_n^{m-1} u )$  & $D_{xx}(U_n^{m-1} u)$ \\ \hline
 2a & $m U_n^{m-1}   D_x u$ & $ D_x ( m U_n^{m-1} D_x u ) $ \\
  2b & &   $m (m-1) U_n^{m-2} \left( D_x u_n \right) \odot \left( D_x u \right)
  + m U_n^{m-1}  D_{xx} u$ \\ \hline
3 &  $D_x(u_n^m) + D_x m U_n^{m-1} (u-u_n)$ &  $ D_{xx}(u_n^m) + D_{xx} m U_n^{m-1} (u-u_n)$ \\ \hline
    \end{tabular}
}







\subsection{Stability analysis for smooth perturbations} \label{sec:Stability}


To understand the impact of the perturbation resulting from replacing 
$f$ by $f^\epsilon$,  let $z_n$ be the numerical solution from the unperturbed DIRK method
\eqref{DIRK} that satisfies the conditions \eqref{Coefficients}:
\begin{eqnarray*} 
    z^{(i)} & =  z_n +  \dt  \sum_{j=1}^{i} a_{ij} f(z^{(j)}), \; \; \;
    z_{n+1}  &=  z_n + \dt \sum_{i=1}^s b_{i} f(z^{(i)}) ,
\end{eqnarray*}
and let $y_n$ be the numerical solution from the perturbed DIRK method \eqref{eq:PerturbedMethod}
\begin{eqnarray*} 
    y^{(i)}  & =  y_n +  \dt  \sum_{j=1}^{i}  \left( a_{ij} f(y^{(j)})
    + a^\varepsilon_{ij} \tau(y^{(j)}) \right),  \; \; \; 
    y_{n+1} & =  y_n + \dt \sum_{i=1}^s b_{i} f(y^{(i)}) ,
\end{eqnarray*}
where $ \tau(u) =  f_{\varepsilon}(u) - f(u) $.


The difference between these two numerical  \eqref{DIRK} and \eqref{MADIRK} 
can be written as a method of the form
\begin{eqnarray*} \label{errors}
 E^{(i)} & =  E_n +  \sum_{j=1}^{i} a_{ij}  \psi^{(j)}
- \dt   \sum_{j=1}^{i} a^\varepsilon_{ij} \tau(y^{(j)}) ,  \; \;  \;
 E_{n+1}  & =  E_n +  \sum_{i=1}^s b_{i} \psi^{(i)}.
\end{eqnarray*}
where
$ E^{(i)} = z^{(i)} - y^{(i)},\; \; \; E_n =z_n - y_n, \; \; \; \psi^{(j)} = \dt \left( f(z^{(i)}) - f(y^{(i)})\right) .$
Correspondingly, we allow  $\vE = \vz - \vy $ and $\Psi $ is the vector of $\psi^{(j)}$.
The following lemmas and theorem  express the impact of the perturbed methods
on the final time errors, which provides an understanding of the stability of this process.
Note that we use the convention that $|\cdot|$ on a matrix or vector is component-wise
absolute value. Also,  while the proofs use a scalar  ODE for simplicity,
the extension to a system is simple (though messy) with the use of Kronecker products.
The proofs follow closely those in \cite{Driscoll2026}, 
and are given Appendix A.1.

\begin{lem} \label{lem0}
    Given an ODE of the form \eqref{ODE} where $f$ is contractive, 
    and a perturbed method of the form \eqref{eq:PerturbedMethod} 
    where the coefficients $\mA= \Aep + \tilde{\mA}$ and $\vb$
    satisfy \eqref{Coefficients}, and 
    we have $z_n$ and $y_n$ coming from the methods above, we can bound the growth
    of the errors at each time level:
\[ \| E_{n+1}\|^2 \leq  \| E_n\|^2   + 2 \dt \left\| \vb^T  \Aep    \Psi  \tau(\vy)   \right\|
 \]
\end{lem}

\begin{lem} \label{lem1}
    Under the conditions in Lemma \ref{lem0} above we can bound the internal
    stage errors by 
    \begin{eqnarray*}
\|  \vE  \| & \leq  & {C}_1   \left\|  E_n      \right\| + \dt  {C}_2   \left\|   \tau(\vy)    \right\| 
\end{eqnarray*}
where
\[ {C}_1 = \left(\mI - \left| \mI - \mAh  \mA^{-1} \right|\right)^{-1} \left|   \mAh \mA^{-1} \right| \ve, \; \; \; \; {C}_2 =  { \left(\mI - \left| \mI - \mAh  \mA^{-1} \right|\right)^{-1}  }  \left|   \mAh \mA^{-1}  \Aep\right| \ve .\]
\end{lem}
Here, $\mAh$ is a diagonal matrix containing the diagonal of $\mA$.

\begin{thm} \label{thm:stability}
Under the conditions in Lemma \ref{lem0} above 
we can bound the  growth in the distance between the two solutions for sufficiently large
$\dt$, and therefore the final time error:
\begin{itemize}
    \item  If $f_\varepsilon$ satisfies the first local consistency conditions
    \eqref{eq:Tau} 
    \[ \|E_{n+1}\|   \leq  \| E_n\|   + \mathcal{K}_1 L \dt^3 .\]
    So that at the final time $T_f$ we have  $ \| E_n\|  \leq \mathcal{K}_1 L T_f \dt^2 .$
    \item  If $f_\varepsilon$ satisfies both local consistency conditions
    \eqref{eq:Tau} and \eqref{eq:TauPrime}
    \[ \|E_{n+1}\|   \leq    \| E_n\| +  \mathcal{K}_2 L \dt^4 .\]
    So that at the final time $T_f$ we have  
$ \| E_n\|  \leq \mathcal{K}_2 L T_f \dt^3 .$
In some cases, $\mathcal{K}_1$ and $\mathcal{K}_2$ may be proportional to
the square of the stiffness of the problem.
\end{itemize}

\end{thm}


\section{New methods for smooth perturbations} \label{sec:NovelMethods}


Using the conditions in Table \eqref{tab:OC} we are able to find the following methods.
The methods of order three and four satisfy conditions \eqref{Coefficients} and so
are B-stable.

\noindent{\bf A2s3p3m: A two stage third order method.} This method is 
based on the coefficients $\mA$ and $\vb$ from 
two stage third order SDIRK method in \cite{kurdi,norsett,CarpenterRK}
given by the parameter $\gamma = \frac{\sqrt{3} + 3}{6}$:
\begin{align} \label{eq:A2s3p3m}
\Aep = \left( \begin{array}{ll} 
   \gamma    &               0 \\
  -1   & \gamma \\
\end{array} \right), \;  
\tilde{\mA} =   \left( \begin{array}{ll}    
0     &        0 \\
2(1 - \gamma )    &              0 \\
\end{array} \right) ,
\vb = \left( \begin{array}{l} 
  1/2 \\ 
 1/2  \\
\end{array} \right).
\end{align}
This method is third order for all smooth perturbations 
that satisfy the first  local consistency condition $\tau(u_n) = 0$ (Equation \eqref{eq:Tau}),
without requiring that the perturbation satisfies the 
additional local consistency condition,
so that we may have $\tau_y(u_n) \neq 0.$

\noindent{\bf A4s4p4m: A four stage fourth order method.}
The four stage method given by the parameter
$ \alpha = \frac{2}{\sqrt{3}} \cos(\frac{\pi}{18})$
\begin{align} \label{eq:A4s4p4m}
\Aep &= \left( \begin{array}{cccc} 
 \frac{1}{2}  & 0  & 0 & 0 \\
    -1 & \frac{1}{2}  & 0  & 0  \\ 
    -2  & \frac{2 \sqrt{3}-1}{2}  & \frac{1}{2}  & 0  \\ 
    \frac{2 \sqrt{3}}{3}   &  -2  &  -\frac{\sqrt{3}}{3}  & \frac{1}{2}   \\ 
\end{array} \right) \; , \;
\tilde{\mA} =  \left(\begin{array}{cccc} 
0           & 0             & 0         &  0        \\
\frac{1}{2} & 0             & 0         &  0        \\
\frac{3}{2} & \frac{3}{2} - \sqrt{3}    &  0   &  0 \\
0           & 2 -  \frac{\sqrt{3}}{3} & 0 & 0 \\
\end{array} \right) , \; \; 
\vb = \left( \begin{array}{c} 
  \frac{8-\sqrt{3}}{12} \\  \frac{1}{6} \\ \frac{1}{6}  \\ 
\frac{\sqrt{3}}{12} \\
\end{array} \right) .
\end{align}
This method is fourth  order for all smooth perturbations that satisfy the first  
local consistency condition $\tau(u_n) = 0$ (Equation \eqref{eq:Tau}),
without requiring that the perturbation satisfies the 
second local consistency condition. This means that we may have $\tau_u(u_n) \neq 0.$ 


\smallskip

If we require that the perturbed operator $f_\varepsilon$ is designed such that the second 
local consistency condition \eqref{eq:TauPrime} is also satisfied, 
then the order conditions simplify significantly. 
In this case, we are able to obtain a three stage fourth order method.

\noindent{\bf B3s4p4m: Three stage fourth order method:}
The method given by the parameter
$\alpha = \frac{2}{\sqrt{3}} \cos(\frac{\pi}{18}) $
and $\beta = 9 \alpha^5 + 12 \alpha^4 - 10 \alpha^3 - 16 \alpha^2 - 3 \alpha$,
\begin{align} \label{eq:B3s4p4m}
    \mA &=  \left( \begin{array}{ccc}
\frac{1+\alpha}{2} & 0 & 0 \\
- \frac{\alpha}{2} &  \frac{1+\alpha}{2} & 0 \\
(1+\alpha) &  -  (1+2 \alpha) & \frac{1+\alpha}{2} \\
\end{array} \right) ,
\vb = \left( \begin{array}{c}
\frac{1}{6 \alpha^2} \\
1 - \frac{1}{3 \alpha^2}\\  \frac{1}{6 \alpha^2} \\
\end{array} \right) , \nonumber \\
\Aep &= \left( \begin{array}{ccc} 
\frac{1+\alpha}{2 } &  0 & 0 \\
         1 - \frac{3}{2} \alpha  & \frac{1+\alpha}{2} & 0\\
    2  & \beta  & \frac{1+\alpha}{2} \\ 
\end{array} \right) ,  \;
\tilde{\mA} =  
\left( \begin{array}{ccc} 
0 & 0 & 0 \\
\alpha - 1 &  0  & 0 \\
\alpha - 1  &  -  (1+2 \alpha + \beta ) & 0  \\
\end{array} \right) .
\end{align}
This method is fourth order if  the perturbation satisfies
both local consistency conditions
\eqref{eq:Tau} and \eqref{eq:TauPrime}, so that
$\tau(u_n) = \tau_u(u_n) =0 $. If the second condition \eqref{eq:TauPrime} 
is not satisfied  then we expect to see  only second order.

\smallskip

\noindent{\bf B6s5p5m: six stage fifth order method}
Fifth order DIRK methods cannot be B-stable \cite{CarpenterRK}, 
and so we do not expect the results of 
Theorem \ref{thm:stability} to hold. 
However, we found the following fifth order method which  A-stable and 
satisfies the perturbation conditions to fifth order if the perturbation satisfies
the  local consistency  conditions \eqref{eq:Tau} and \eqref{eq:TauPrime}), so that 
$\tau(u_n) = \tau_u(u_n) = 0$.  The coefficients of this method are given in Appendix
\ref{app:5thO}. In Section \ref{sec:NumericalTests} we will test this method
and study its stability properties in practice.


   \begin{rmk}
      We use the following notation: Methods that that have a diagonal $\Aep$ are designated with a `D'.
      Methods that require the local consistency condition \eqref{eq:Tau} are designated with a `A'.
       Methods that require both local consistency conditions \eqref{eq:Tau} and \eqref{eq:TauPrime}
       are designated with a `B'. The rest of the name is in the format SsPpMm where S is the number of stages,
       P is the order of the method, M is the perturbation order.
  \end{rmk}


\section{Numerical tests} \label{sec:NumericalTests}



In this section we consider the methods {\bf A2s3p3m} \eqref{eq:A2s3p3m}, 
{\bf A4s4p4m} \eqref{eq:A4s4p4m},
{\bf B3s4p4m} \eqref{eq:B3s4p4m}, 
and {\bf B6s5p5m} \eqref{eq:B6s5p5m}, and the following diagonally perturbed DIRK methods:
\begin{enumerate}
\item {\bf D1s2p1m (IMR)} The diagonally perturbed second order implicit midpoint rule 
\begin{eqnarray} \label{pIMR}
& u^{(1)}  =  u_n + \frac{1}{2} \dt f_\varepsilon (u^{(1)})  , \; \; 
& u_{n+1}  =  u_n + \dt f(u^{(1)}) .
\end{eqnarray}
     \item {\bf D2s3p1m  (SDIRK3)} 
The diagonally perturbed third order singly diagonally implicit method  given by 
\begin{eqnarray} \label{pSDIRK3} 
u^{(1)} & = & u_n + \gamma\dt f_\varepsilon (u^{(1)})   \nonumber \\ 
u^{(2)} & = & u_n + (1-2\gamma)\dt  f(u^{(1)}) + \gamma\dt f_\varepsilon (u^{(2)})    \nonumber \\
u_{n+1} & = & u_n + \frac{\dt}{2}f(u^{(1)}) + \frac{\dt}{2}f(u^{(2)}),
\end{eqnarray}
where $\gamma = \frac{\sqrt{3} + 3}{6}$.
This method does not need $\tau(u_n) = 0$, but we expect it to perform
as $O(\dt^3) + O(\dt \tau)$. This method is the diagonally perturbed version 
of the same method that resulted in the perturbed method {\bf A2s3p3m} \eqref{eq:A2s3p3m}.
\item {\bf D3s4p1m  (SDIRK4)} The perturbed fourth order pSDIRK4 given by  $\alpha = \frac{2}{\sqrt{3}} \cos(\frac{\pi}{18}) $:
\begin{eqnarray} \label{pSDIRK4}
u^{(1)} & = & u_n + \frac{1+\alpha}{2} \dt f_\varepsilon(u^{(1)})  \nonumber \\ 
u^{(2)} & = & u_n - \frac{\alpha}{2} \dt  f(u^{(1)}) 
+ \frac{1+\alpha}{2} \dt f_\varepsilon(u^{(2)}) \nonumber\\
u^{(3)} & = & u_n + (1+\alpha) \dt  f(u^{(1)}) -  (1+2 \alpha)  \dt f(u^{(2)}) 
+\frac{1+\alpha}{2} \dt f_\varepsilon(u^{(3)})  \nonumber \\
   \; \; \; \;  u_{n+1} & = & u_n + \frac{\dt}{6 \alpha^2} \big(f(u^{(1)}) 
    + (6 \alpha^2 -2) f(u^{(2)}) + f(u^{(3)})\big) ,
\end{eqnarray}
Note that when $f_\varepsilon = f$ this method becomes the usual SDIRK4 method of
\cite{norsett,CrouzeixBstab}.
 \end{enumerate}
 These are diagonally perturbed methods we previously studied, and in \cite{Driscoll2026}
 we showed that their performance can be improved with the use of specially designed 
 stabilized corrections. However, in this work we do not correct these methods.
 
We consider three numerical examples, where for each 
the perturbed function $f_\varepsilon$ results from different  linearizations.
We show numerically that the theory and novel methods developed in 
Sections \ref{sec:MADIRK}  and \ref{sec:NovelMethods} perform as expected.
We also demonstrate that the properties of the linearization 
impact the performance of the methods as expected.

 
\subsection{Inviscid Burgers' equation}


Consider the inviscid Burgers' equation
 $ \mathcal{U}_t + \left( \frac{1}{2} \mathcal{U}^2 \right)_x =0,$
 on the domain $x = [0,2\pi)$, with initial condition $\mathcal{U}(x,0)= \frac{1}{2} + \frac{1}{4}\sin(x)$ 
 and periodic boundary conditions. We are interested in the solution of this equation up to final time 
 $T_f = 3.5$, which is before the shock forms.
 We semi-discretize this equation using a Fourier spectral method differentiation 
 matrix $D_x$, resulting in the system of ODEs 
 $ u' = f(u) = - \frac{1}{2} D_x u^{2}.$
 We then evolve this forward with the time-evolution methods  listed above.
In this section we verify the performance of the novel methods by comparing them
to previously used diagonally perturbed methods, 
and investigate the impact of different linearizations 
on the different perturbed methods.

The first linearization $f_\varepsilon^1$ of $f$
is obtained by linearizing $f$ around the point $\bar{y} = u_n$
and then differentiating:
 \begin{eqnarray}\label{linB1}
   f_{\epsilon}^1(\bar{y},y) &=& - \frac{1}{2} D_x \bar{Y} y, 
   \end{eqnarray}
   where $\bar{Y} = \text{diag}(\bar{y})$ denotes the diagonal matrix with $\bar{y}$ on its diagonal. 
The perturbation for this form is
  \[   \tau^1(\bar{y},y) =  f(y) - f_{\epsilon}^1(\bar{y},y) = 
  - \frac{1}{2} D_x Y y + \frac{1}{2} D_x \bar{Y} y 
=  - \frac{1}{2} D_x ( Y -\bar{Y} ) y .
\]
This is $ O(\dt)$ if    $\left| y - \bar{y} \right| = O(\dt)$.
It is clear that if $\bar{y} = u_n$ then  $\tau^1(\bar{y},u_n) = 0$.

The second linearization 
is obtained by differentiating first and then linearizing:
     \begin{eqnarray}\label{linB2}
     f_{\epsilon}^2(\bar{y},y) &=& - \bar{Y} D_x  y.
    \end{eqnarray}
The perturbation     resulting from this linearization:
\[ \tau^2(\bar{y},y) = f(y) - f_{\epsilon}^2(\bar{y},y) 
=   - \frac{1}{2} D_x Y y   + \bar{Y} D_x  y 
= (Y D_x  - \frac{1}{2} D_x Y) y +  (\bar{Y} -Y) D_x y,
\]
while the second part is clearly $O(\dt)$ if $\left| y - \bar{y} \right| = O(\dt)$, 
the first part depends on the  spatial refinement rather than the time-step. 
For this reason, we do not expect that $\tau^2_{\bar{y}}(u_n) $ will be zero for this linearization.
However, as $N_x \rightarrow 0$, we should see $\tau^2_{\bar{y}}(u_n) \rightarrow 0$
if $\bar{y} = u_n$. 

Finally, we  use a Taylor expansion to linearize \eqref{linB3} (green)
\begin{eqnarray}\label{linB3}
f_{\epsilon}^3(\bar{y},y) &=&  f(\bar{y}) + 
    f'(\bar{y}) \left( y - \bar{y} \right) 
    =  - \frac{1}{2} D_x \bar{y}^2
    - D_x \bar{Y} \left( y - \bar{y} \right) .
    \end{eqnarray}
    If $\left| y - \bar{y} \right| = O(\dt)$ then
the perturbation error will be  $O(\dt^2)$ 
\[ \tau^3(\bar{y},y) = f(y) - f_{\epsilon}^3(\bar{y},y) = - \frac{1}{2} D_x Y y +
  \frac{1}{2} D_x \bar{Y} \bar{y}
+ D_x \bar{Y} \left( y - \bar{y} \right) 
=  -  \frac{1}{2} D_x  \left(\bar{Y} -Y
\right)^2 \ve 
\]
  We expect that this linearization, which has a smaller perturbation, 
  to provide enhanced stability and accuracy.
  Furthermore, we note that if $\bar{y} = u_n$ then this linearization satisfies 
both local consistency conditions \eqref{eq:Tau} and \eqref{eq:TauPrime}.

\begin{figure}[htb]
    \centering
    \begin{subfigure}[b]{0.32\textwidth}
        \centering
        \includegraphics[width=\textwidth]{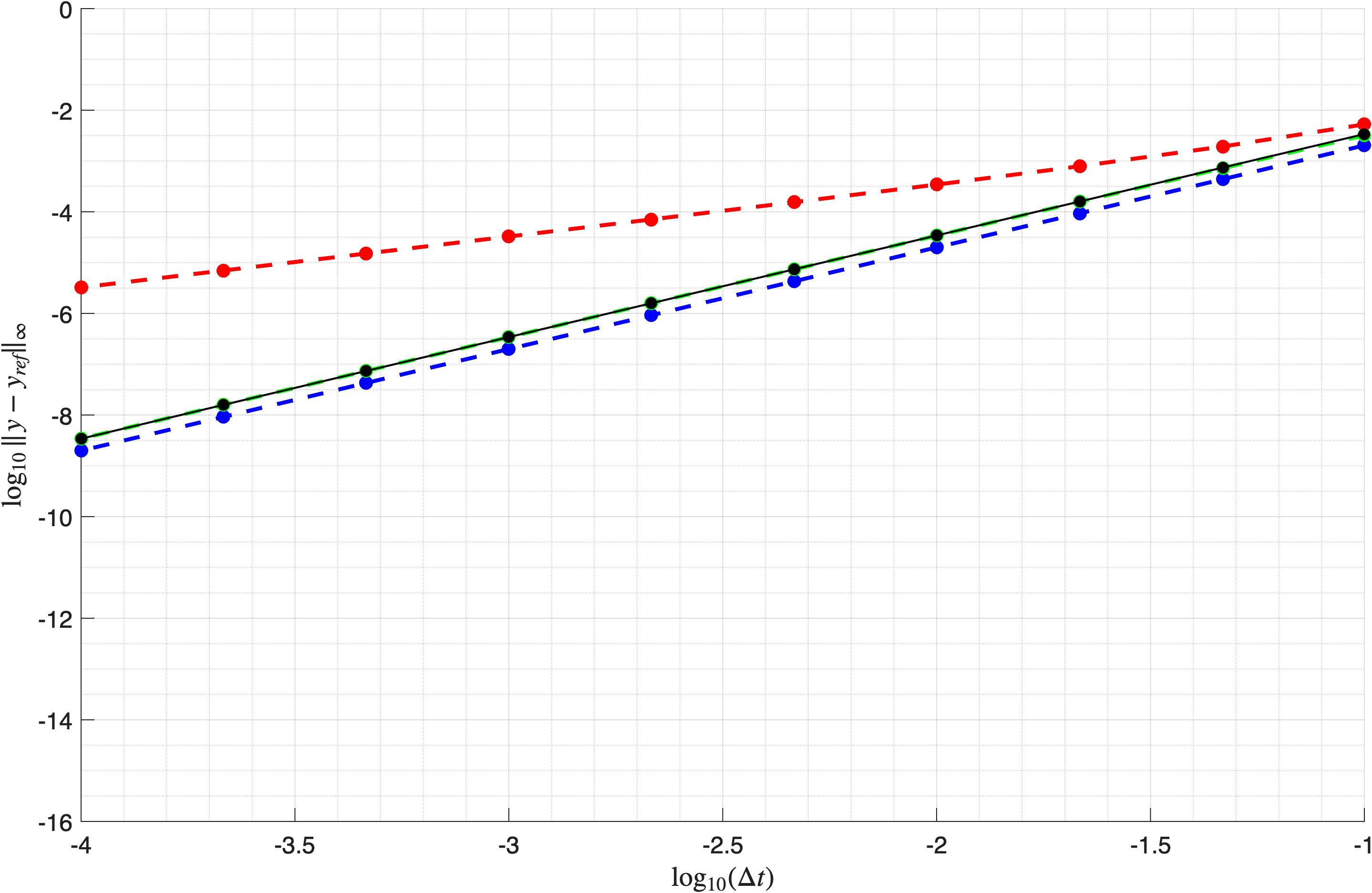}
    \end{subfigure}
    \hfill
    \begin{subfigure}[b]{0.32\textwidth}
        \centering
        \includegraphics[width=\textwidth]{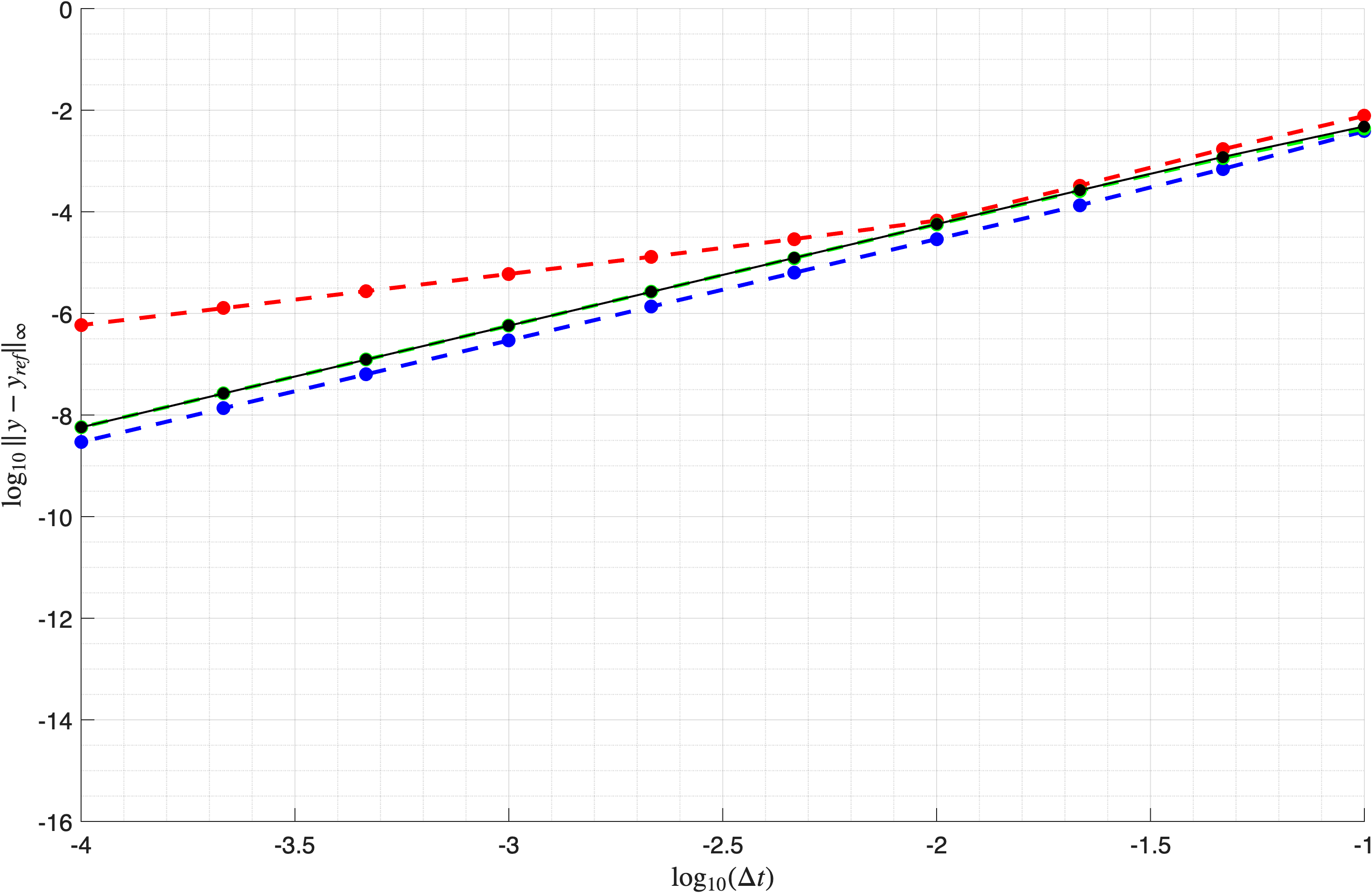}
    \end{subfigure}
    \hfill
    \begin{subfigure}[b]{0.32\textwidth}
        \centering
        \includegraphics[width=\textwidth]{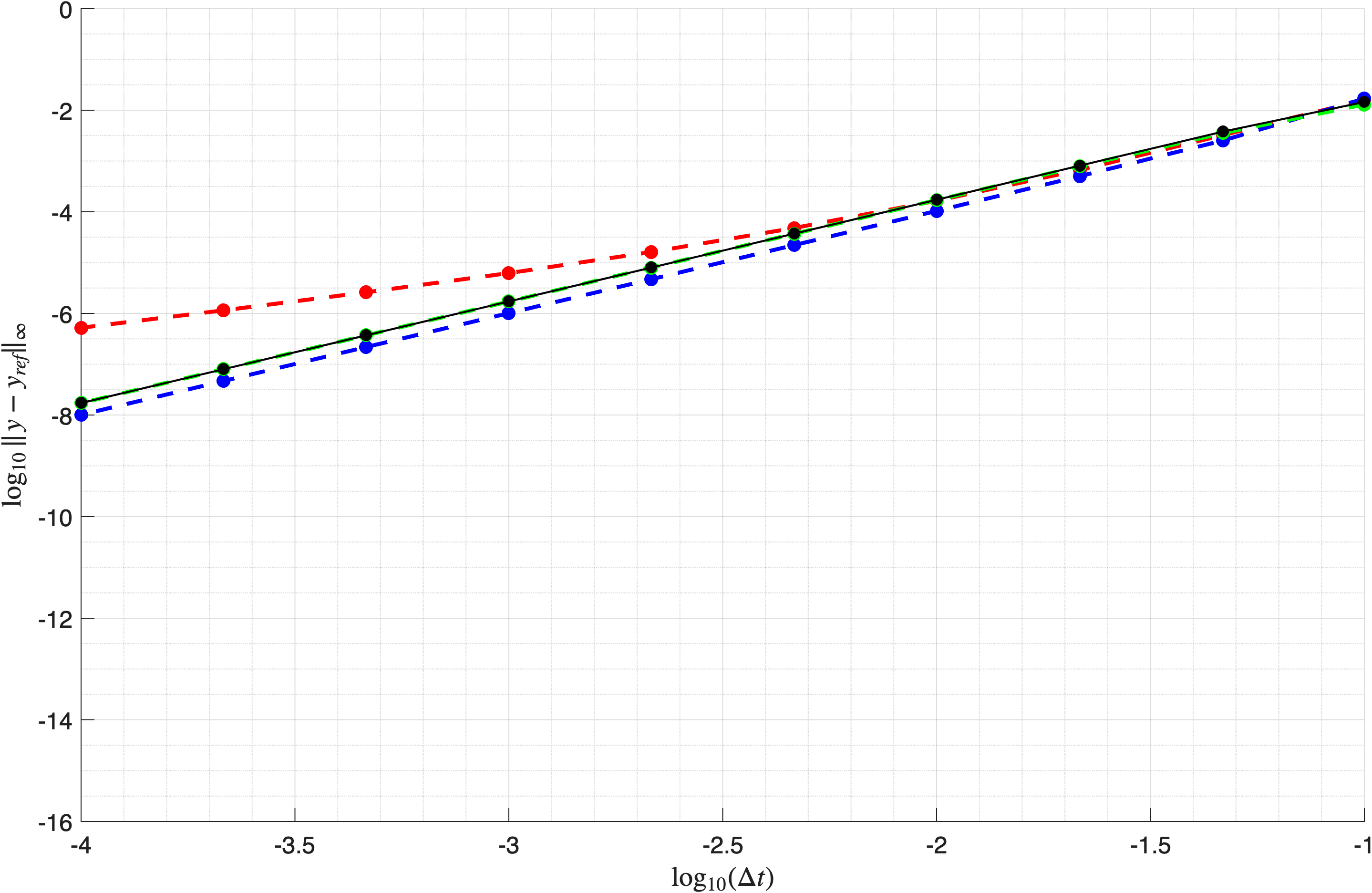}
    \end{subfigure}

    \vspace{1em}

    \begin{subfigure}[b]{0.32\textwidth}
        \centering
        \includegraphics[width=\textwidth]{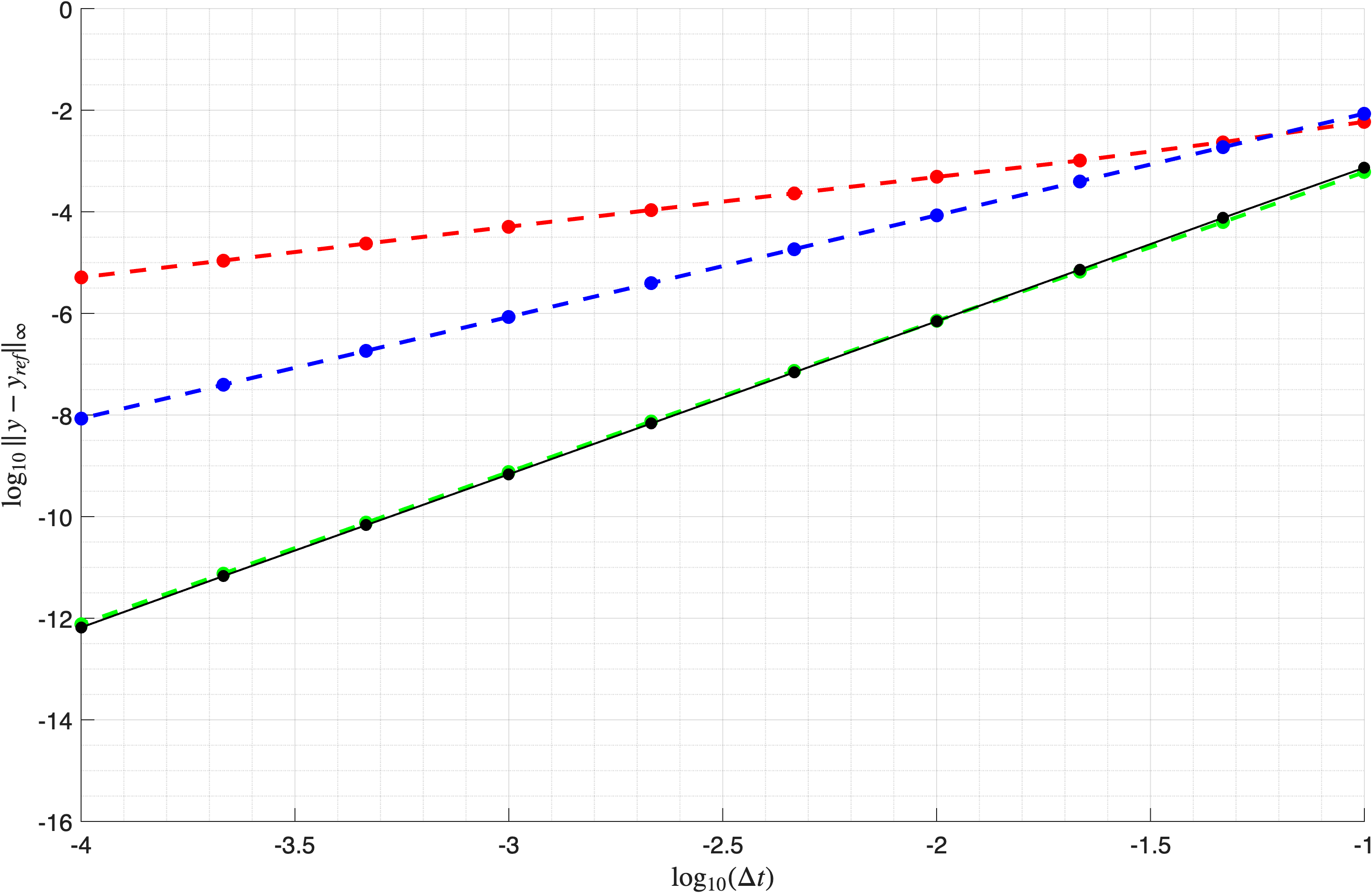}
    \end{subfigure}
    \hfill
    \begin{subfigure}[b]{0.32\textwidth}
        \centering
        \includegraphics[width=\textwidth]{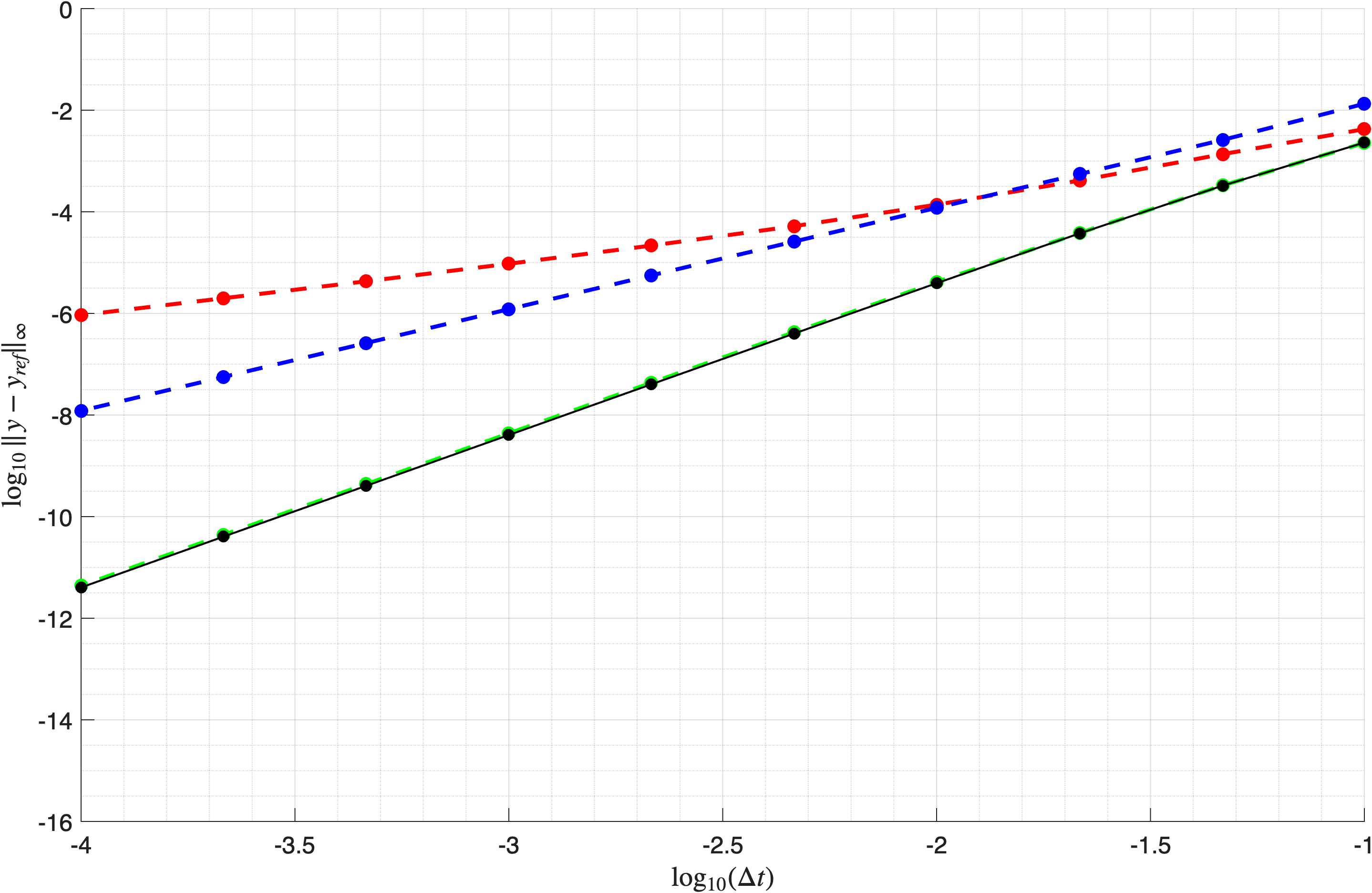}
    \end{subfigure}
    \hfill
    \begin{subfigure}[b]{0.32\textwidth}
        \centering
        \includegraphics[width=\textwidth]{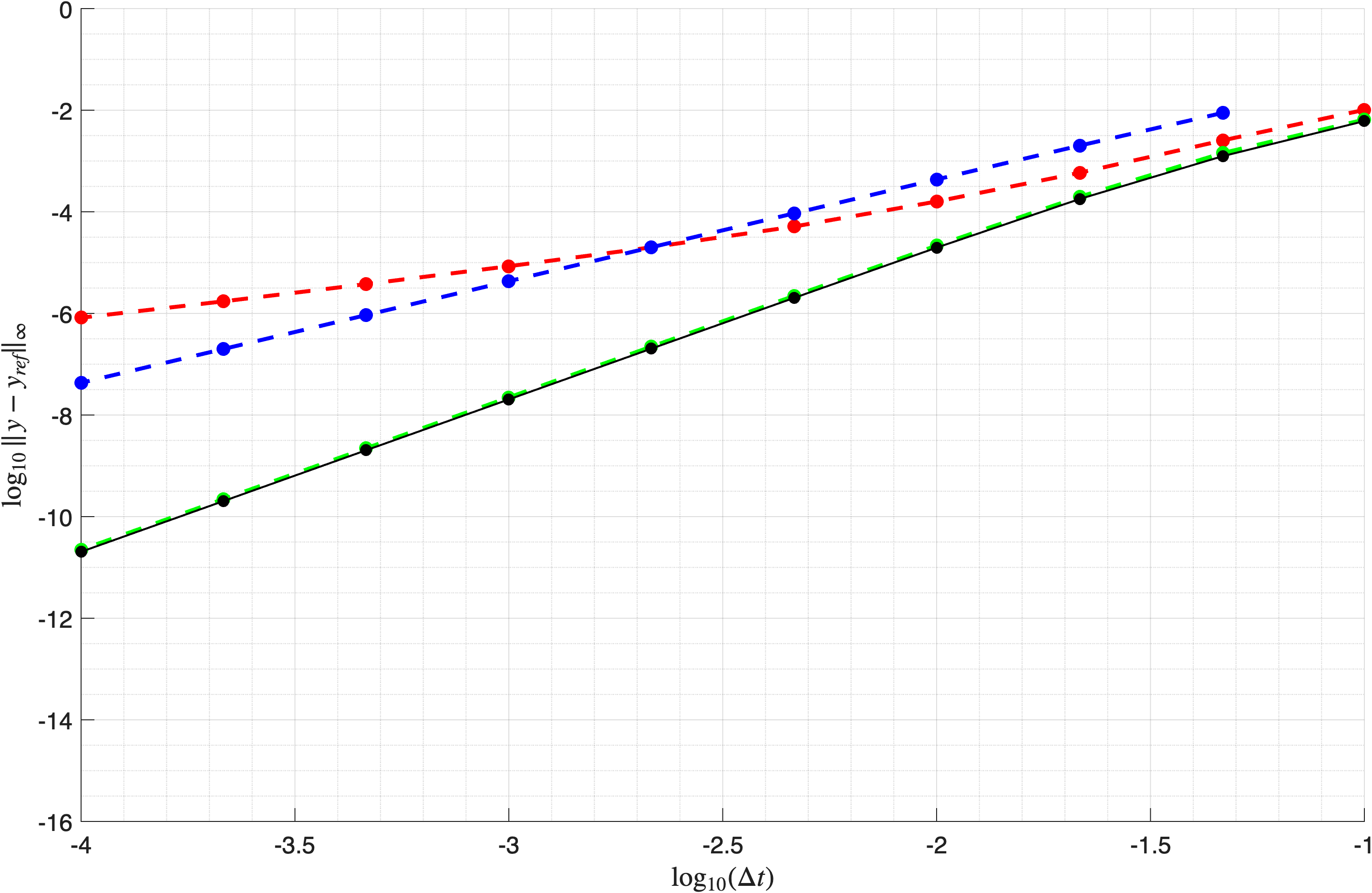}
    \end{subfigure}

    \vspace{1em}

    \begin{subfigure}[b]{0.32\textwidth}
        \centering
        \includegraphics[width=\textwidth]{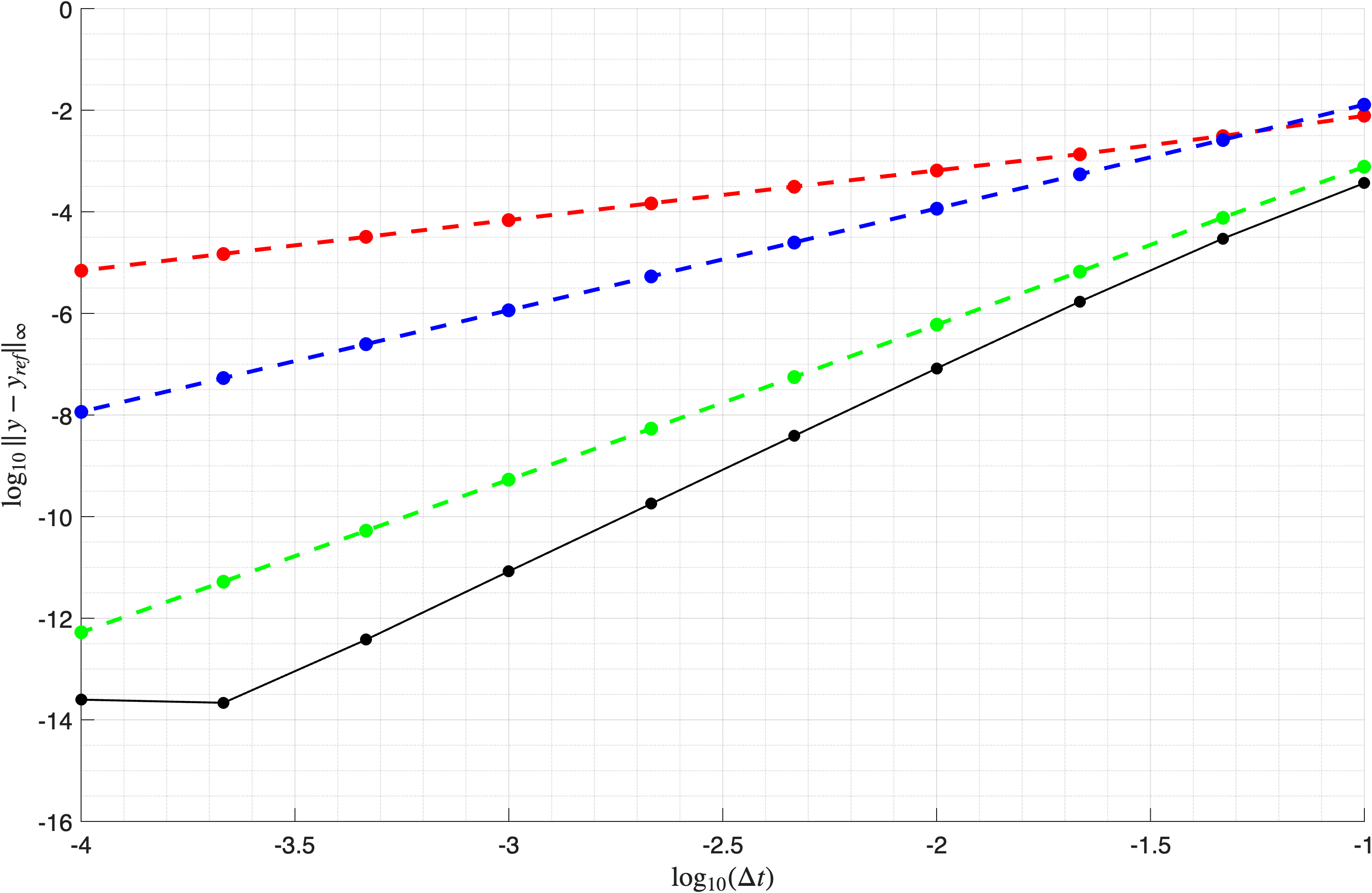}
    \end{subfigure}
    \hfill
    \begin{subfigure}[b]{0.32\textwidth}
        \centering
        \includegraphics[width=\textwidth]{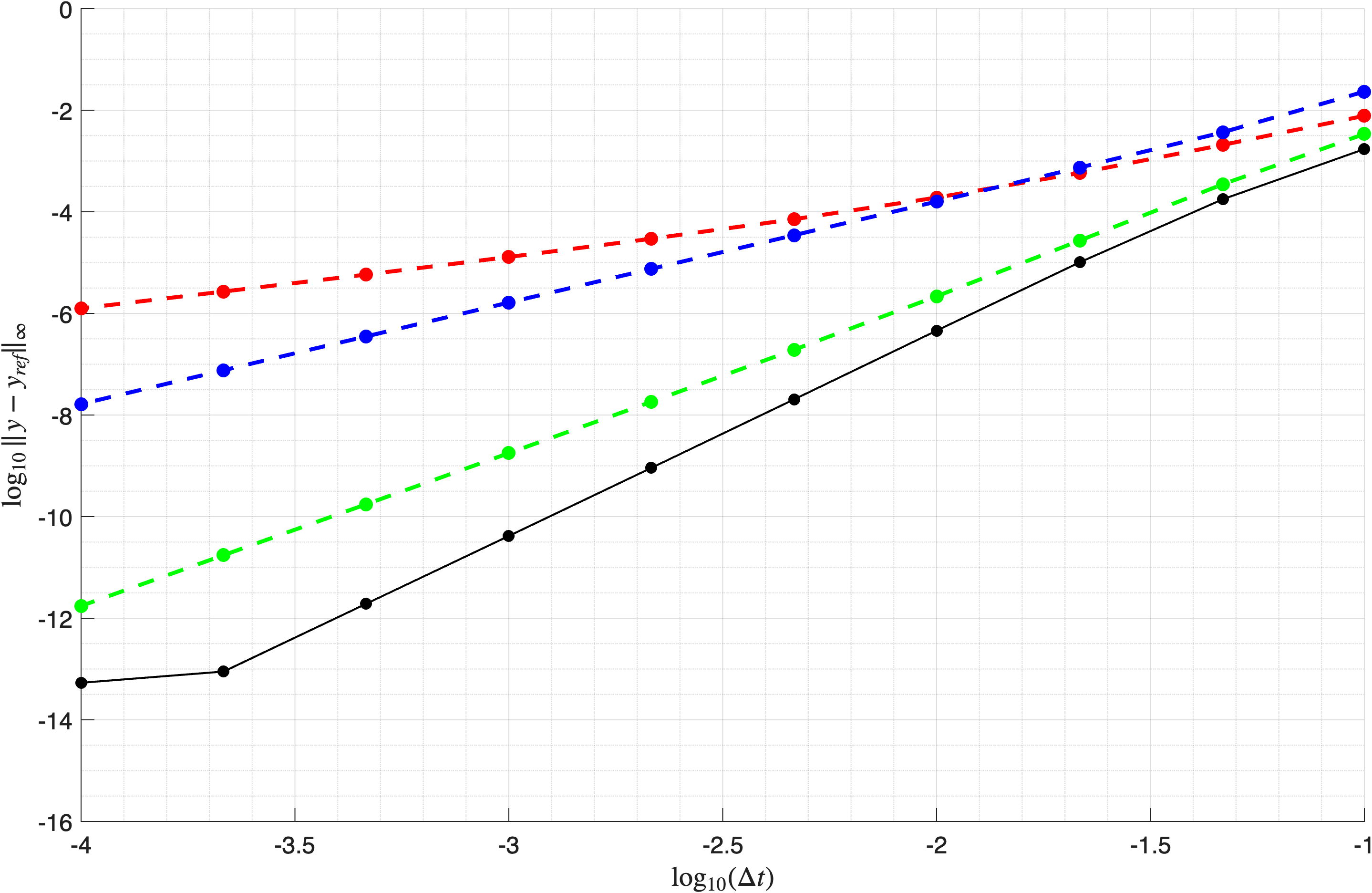}
    \end{subfigure}
    \hfill
    \begin{subfigure}[b]{0.32\textwidth}
        \centering
        \includegraphics[width=\textwidth]{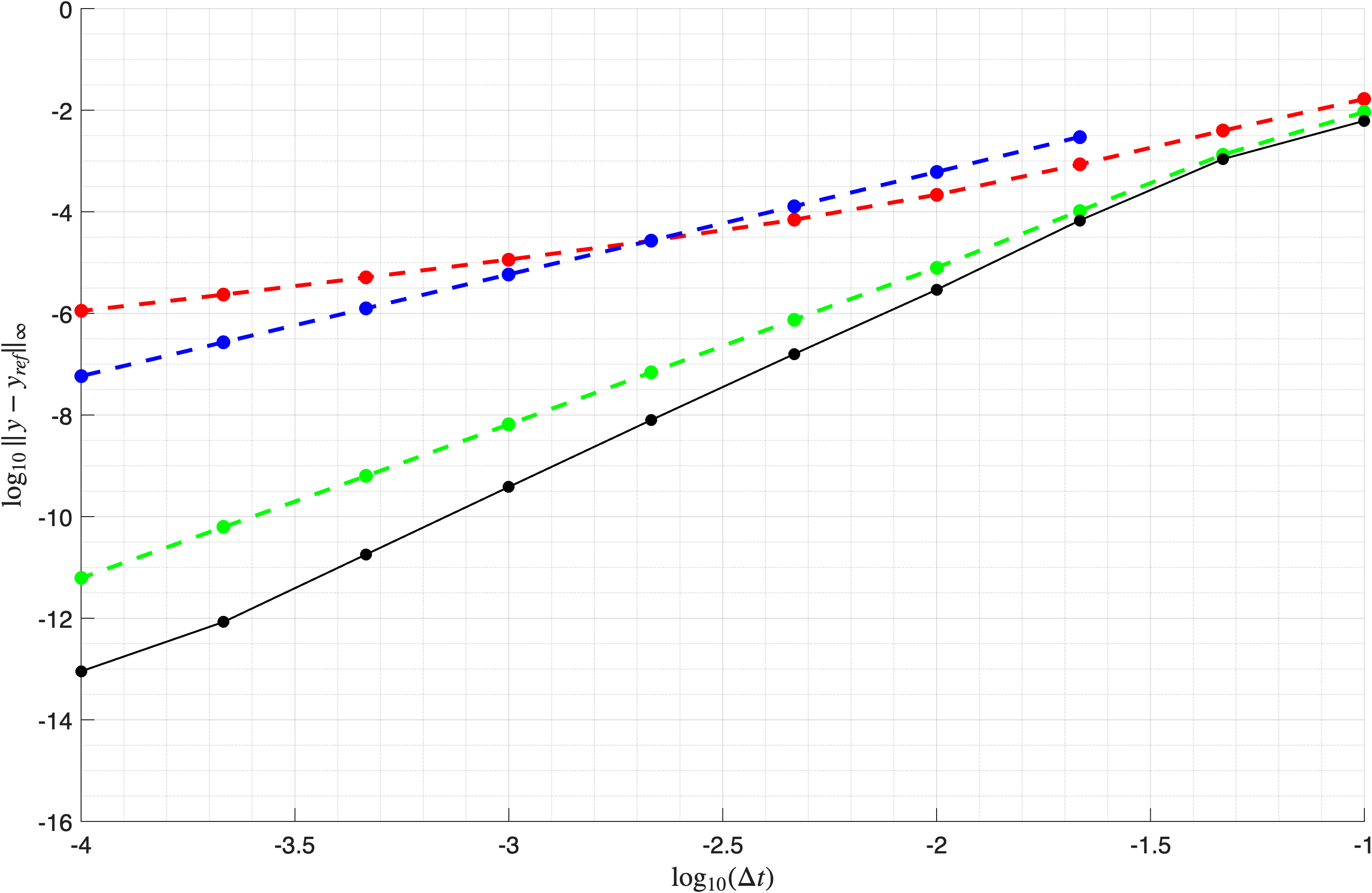}
    \end{subfigure}
    \caption{Burgers' Equation: 
    Final convergence plots for the diagonally perturbed DIRK methods,
    D1s2p1m (IMR) \eqref{pIMR} (top), D2s3p1m (SDIRK3) \eqref{pSDIRK3} (middle), 
    and D3s4p1m (SDIRK4) \eqref{pSDIRK4} (bottom),
    across varying spatial resolutions $N_x=20$ (left),
    $N_x=40$ (middle), and $N_x=100$ (right). The maximum norm errors are measured 
    against a reference solution calculated using 4th order Runge Kutta method, and plotted
    for different values of $\dt$.
    We compare the impact of the 
three linearizations
\eqref{linB1} (blue), \eqref{linB2} (red), and
\eqref{linB3} (green),  for $f_\epsilon$, with $\bar{y} = u_n.$ 
}
    \label{fig:convergence_plots}
\end{figure}

Figure \ref{fig:convergence_plots} focuses on previously studied
diagonally perturbed DIRK methods. The figures show the final time maximum norm  
errors compared to a reference solution for 
the  diagonally perturbed methods D1s2p1m \eqref{pIMR} (top), 
the diagonally perturbed D2s3p1m \eqref{pSDIRK3} (middle),
and the diagonally perturbed D3s4p1m \eqref{pSDIRK4} (bottom) for various time-steps
$\dt$ (shown as $\log_{10}(\dt)$ on the horizontal axis).
In each plot we compare the three linearizations
\eqref{linB1} (blue), \eqref{linB2} (red), and
\eqref{linB3} (green),  for $f_\epsilon$, with $\bar{y} = u_n.$ 
The black line is the result of solving the implicit stages correctly with 
a fully resolved Newton solver as reference. 
We also compare three levels of spatial refinement: $N_x =20$ (left),
$N_x =40$ (middle), and $N_x =100$ (right).

On the top row of Figure \ref{fig:convergence_plots} 
we observe that for the diagonally perturbed second order implicit midpoint rule IMR 
\eqref{pIMR} the consistent linearizations \eqref{linB1} (blue) and \eqref{linB3} (green)
always perform at the target second order.
This is due to the nice property of the linearizations that $\tau(u_n)=0$ thus eliminating the second order perturbation term as noted in Table \ref{tab:OC}. However, the inconsistent linearization
\eqref{linB2} (red) is less accurate  when the spatial refinement is low (for $N_x=20$)
but improved when the number of spatial points increases. 
Again this can be directly traced back to the behavior of the perturbation 
error $\tau^2(u_n,u_n)\neq0$, which is nonzero but decays  with $\Delta x$. When refined enough 
we see the $O(\Delta t^2)$ error of the underlying integration scheme dominate.

In the middle row of Figure \ref{fig:convergence_plots}  
we see that for the third order SDIRK3 \eqref{pSDIRK3}
The inconsistent linearization \eqref{linB2} (red) gives, roughly, first order convergence,
the first order consistent linearization  \eqref{linB1} (blue) gives second order performance,
and the second order consistent Taylor series linearization \eqref{linB3} (green)
gives third order performance. On the other hand, the inconsistent linearization
\eqref{linB2} (red) gives only first order for small number of points, but becomes second
order as the spatial grid is refined. Here we note the added feature of the Taylor series linearization and that is $(\tau^3_u(\bar{y},u_n)=0$, thus eliminating the third order 
perturbation conditions found in Table \ref{tab:OC}, obtaining 
consistent $O(\Delta t^3)$ errors for all time-steps.
Finally, in the bottom row we look at the impact of the linearizations on the 
fourth order SDIRK4 \eqref{pSDIRK4}. The results for the linearization \eqref{linB1} (blue)
and \eqref{linB3} (green) are similar to those of the row above, providing second order
and third order performance (respectively), although the underlying method is fourth order.
This shows clearly that even an excellent linearization will reduce the order of a method
unless it is specially designed to recover higher order accuracy.

In Figure \ref{fig:InviscidBurgers_Perturbed} we turn to the
novel perturbed DIRK methods specially designed to  handle perturbations, 
and compare their performance in the presence of linearizations
to the diagonally-perturbed methods in Figure \ref{fig:convergence_plots}.
The dotted lines represent linearization around 
a different $\bar{y}$  at each stage: $\bar{y} = u^{(i-1)}$.
The dashed lines represent linearization around  $\bar{y} = u_n$. 
The red lines are the results from the  inconsistent linearization \eqref{linB2}. 
The blue lines are the results from the linearization \eqref{linB1} 
that satisfies the local consistency condition \eqref{eq:Tau}. 
The green lines are the results from the linearization \eqref{linB3} 
that satisfies both the local consistency conditions \eqref{eq:Tau}
and \eqref{eq:TauPrime}.

\begin{figure}[htb]
    \centering
    \begin{subfigure}[b]{\textwidth}
        \centering
        \includegraphics[width=0.49\textwidth]{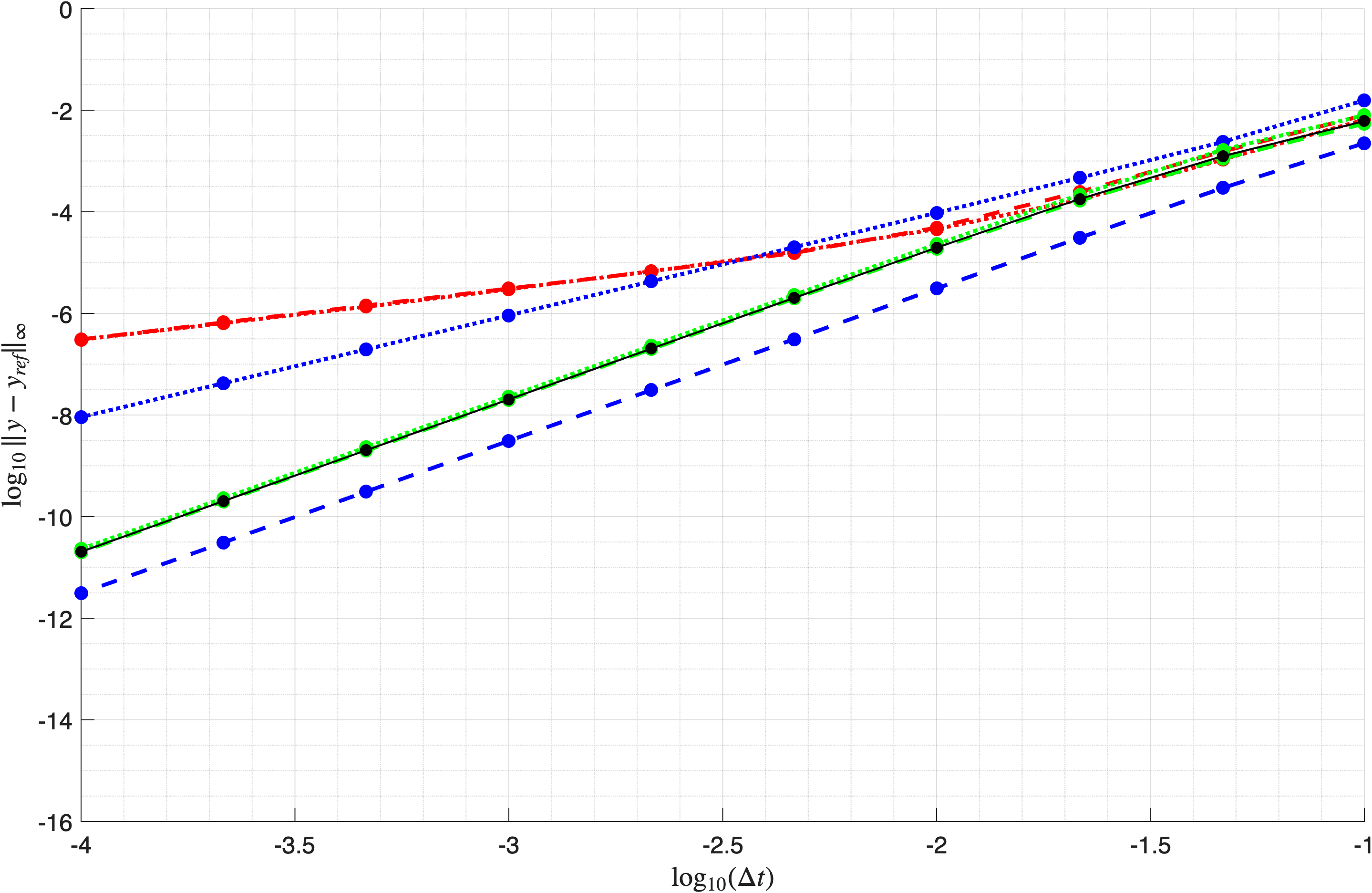}
        \includegraphics[width=0.49\textwidth]{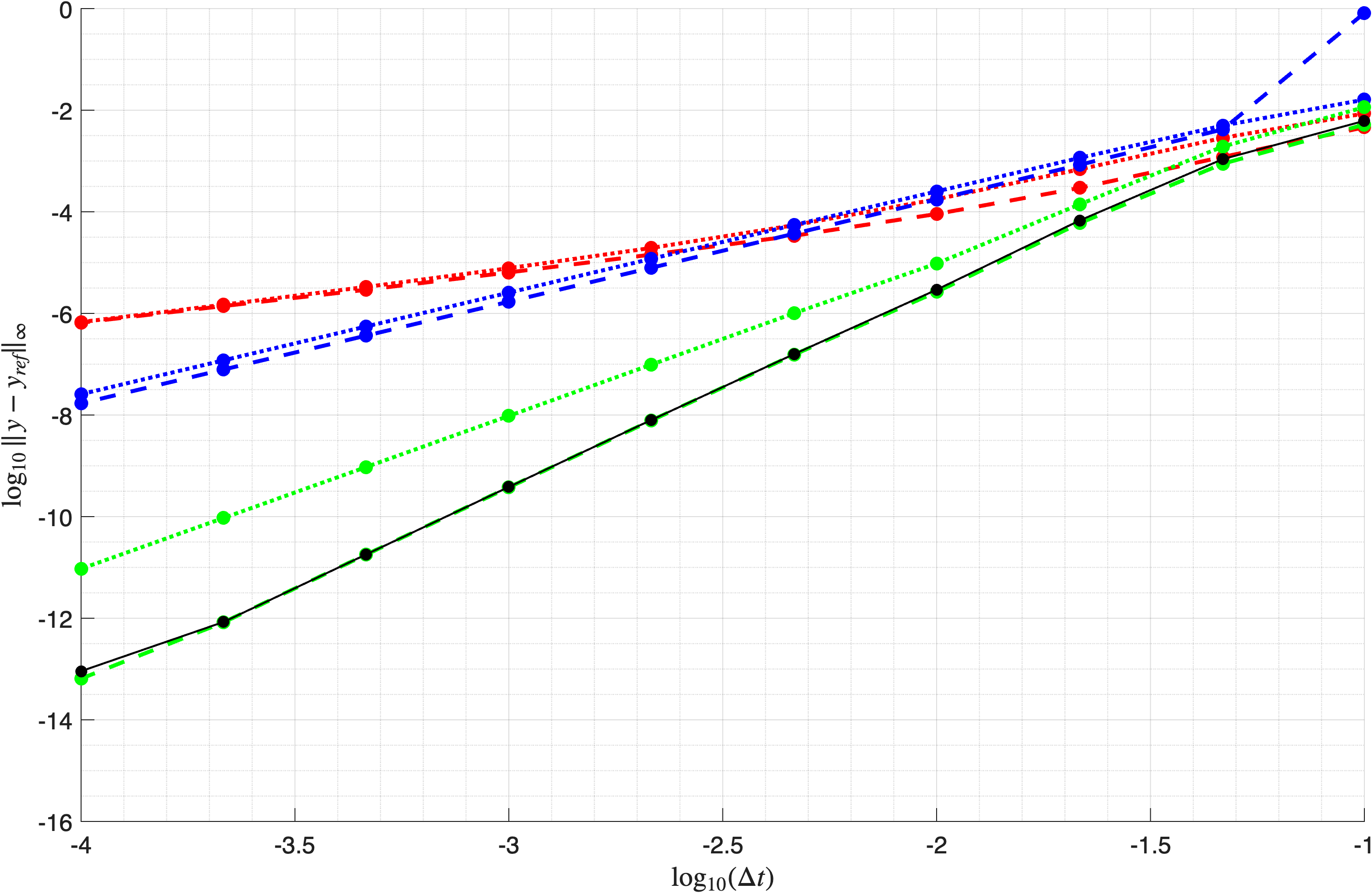}
    \end{subfigure} \\
    \begin{subfigure}[b]{\textwidth}
        \centering
        \includegraphics[width=0.49\textwidth]{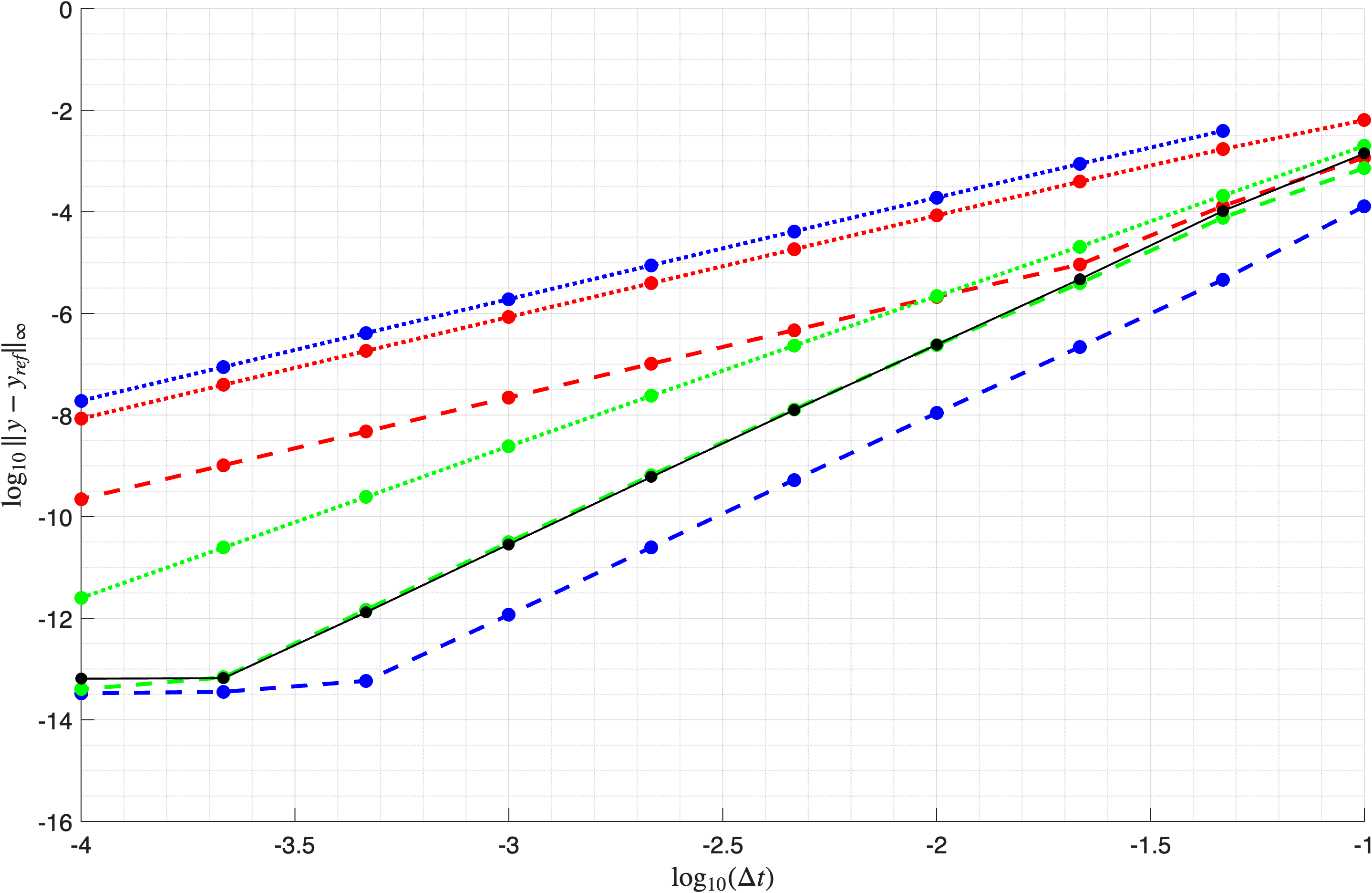}
        \includegraphics[width=0.49\textwidth]{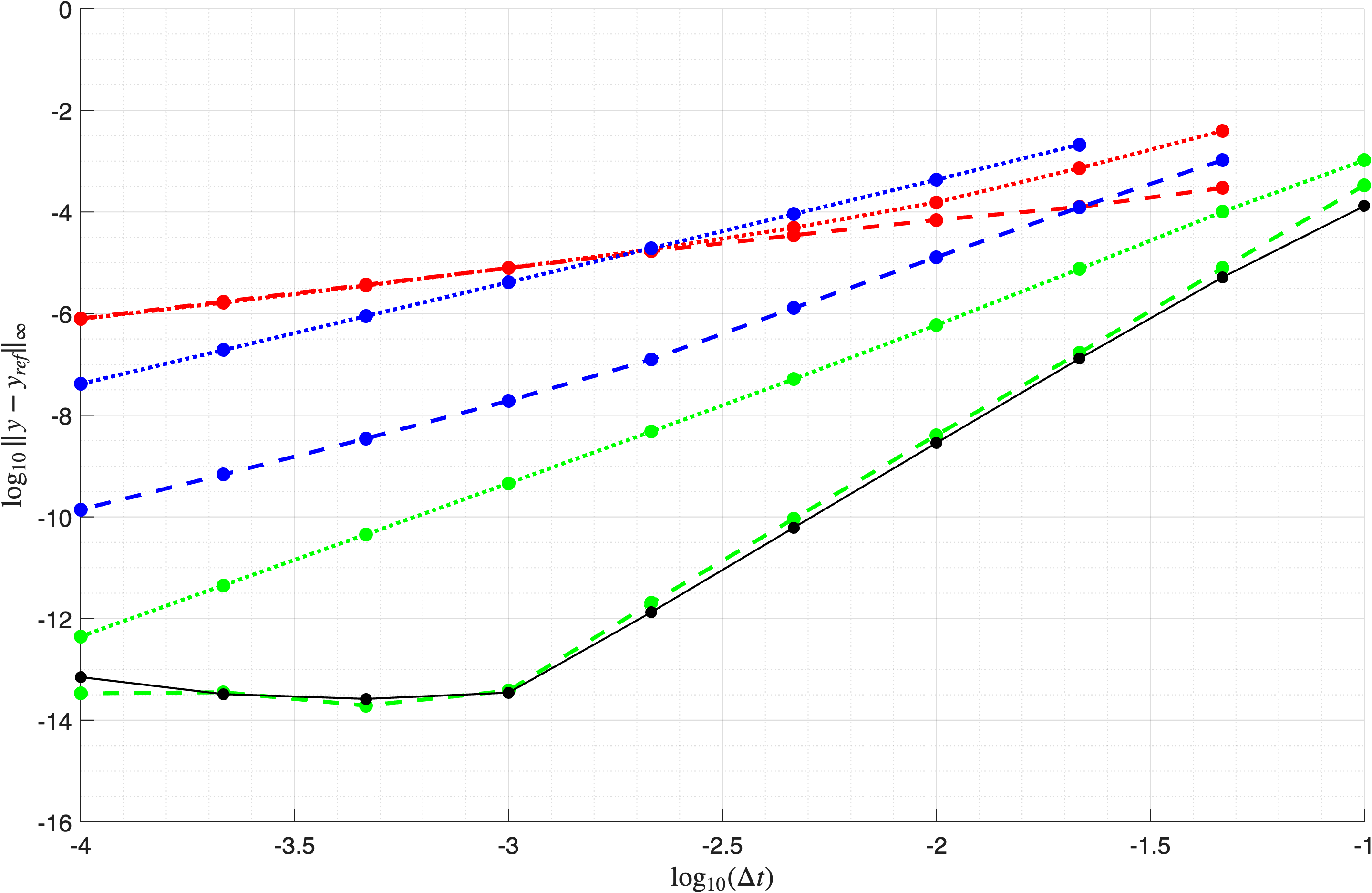}
    \end{subfigure}
    
    \begin{minipage}[c]{0.3\textwidth}  \hspace{0.1in}
        \includegraphics[width=0.85\textwidth]{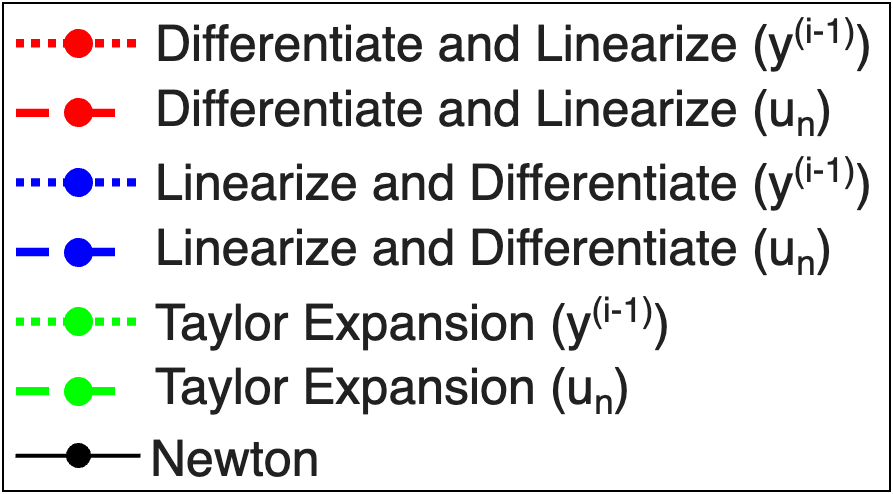}
    \end{minipage} \hfill
    \begin{minipage}[c]{0.65\textwidth} 
        \caption{Burgers' equation:  final time errors using the novel perturbed SDIRK methods:
         \\
        Top Left: A2s3p3m \eqref{eq:A2s3p3m} \\
         Bottom Left: A4s4p4m \eqref{eq:A4s4p4m} \\
        Top Right: B3s4p4m \eqref{eq:B3s4p4m} \\
        Bottom Right: B6s5p5m \eqref{eq:B6s5p5m}.\\
        $N_x = 101$ spatial points.
        \label{fig:InviscidBurgers_Perturbed}} 
    \end{minipage} \hfill 
\end{figure}
On the top left of Figure \ref{fig:InviscidBurgers_Perturbed}
we see the results from the  third order A2s3p3m  method \eqref{eq:A2s3p3m} 
with different linearizations.  For the inconsistent linearization \eqref{linB2}  (red)
there is no difference whether the linearization is around $u_n$ or $u^{(i-1)}$.
This is because the linearization is inconsistent (except as $N_x \rightarrow 0$)
so  will perform at  first order regardless of $\bar{y}$.
For the consistent Taylor series linearization \eqref{linB3} there is also no difference
whether the linearization is around $u_n$ or $u^{(i-1)}$, because the linearization error is such that 
the method  perform at third order as long as  $\bar{y}$ is close
enough to $u_n$. However, for  the linearization \eqref{linB1}  (blue)
the linearization at $\bar{y} = u_n$ (dashed lines) results in higher order 
(third order) and smaller  errors than using  $\bar{y} = u^{(i-1)}$ at each stage 
(dotted lines). 

To understand the impact of the point $\bar{y}$, 
consider that the B-series of this method contains the term 
$\Delta t^2 \vb \cep f_u(u_n) \tau(u^{(i-1)}, u_n)\neq 0$. 
 Although we cannot eliminate this term, we  still benefit from 
 the special cancellations up to a point.  
 For linearization \eqref{linB1} we can expect  $\tau(u^{(i-1)},u_n)=O(\Delta t)$ 
 which explains the added order of accuracy, with the local perturbation term 
 becoming $O(\Delta t^3)$ locally therefore $O(\dt^2)$ globally.  
 For linearization \eqref{linB3} we also observe interesting behavior:
\begin{eqnarray*}
    \tau^3(u^{(i-1)},u_n) & = & f_\varepsilon(u^{(i-1)},u_n) - f(u_n) \\
    & = & f(u^{(i-1)})  + f'(u^{(i-1)})  (u_n - u^{(i-1)} ) - f(u_n) \\
    & = & f\left(u_n + \dt \sum_{j=1}^{i-1} a_{i-1,j} f(u^{(j)}) \right)  
    - f'(u^{(i-1)})  (\dt \sum_{j=1}^{i-1} a_{i-1,j} f(u^{(j)}) ) - f(u_n) \\
    & = & f(u_n) + \dt \sum_{j=1}^{i-1} a_{i-1,j} f(u^{(j)}) f'(u_n) + O(\dt^2)  \\
    &&   - (\dt \sum_{j=1}^{i-1} a_{i-1,j} f(u^{(j)}) ) f'(u^{(i-1)})  - f(u_n) \\
   & = & \dt \sum_{j=1}^{i-1} a_{i-1,j} f(u^{(j)}) f'(u_n) -
    \dt \sum_{j=1}^{i-1} a_{i-1,j} f(u^{(j)})  f'(u_n)
   + O(\dt^2)  \\
   & = & O(\dt^2).
\end{eqnarray*}
So we can expect  $\tau(u^{(i-1)},u_n)=O(\Delta t^2)$, and similarly we can expect
 $\tau_u(u^{(i-1)},u_n)=O(\Delta t)$. 
 This ensures the second order perturbation term is damped by $O(\Delta t^2)$ 
 and the third order perturbation terms is damped by at least $O(\Delta t)$.

On the bottom left of Figure \ref{fig:InviscidBurgers_Perturbed} 
are the results from the  A4s4p4m method \eqref{eq:A4s4p4m},
We see that the inconsistent linearization \eqref{linB2}  results in second order 
errors regardless of the value it is linearized around (red lines).  
The linearization \eqref{linB1} will satisfy the local consistency condition
\eqref{eq:Tau} when linearized around $\bar{y} = u_n$, and so the 
A4s4p4m method \eqref{eq:A4s4p4m} performs as designed at fourth order in this 
case (blue dashed line), and in fact has the smallest errors. 
However, when the linearization is performed around $\bar{y} = u^{(i-1)}$ 
it no longer satisfies the local consistency condition
\eqref{eq:Tau}  and so the convergence rate drops to second order (blue dotted line).
Finally, for the Taylor series linearization \eqref{linB3} we see third order when 
$\bar{y} = u^{(i-1)}$ (green dotted line)
and, as expected, fourth order when $\bar{y} = u_n$ (green dashed line). Note that 
this performs exactly as the Newton iteration would for the unperturbed base method
(black line).

On the top right of Figure \ref{fig:InviscidBurgers_Perturbed}  
we have the results from the B3s4p4m method \eqref{eq:B3s4p4m},
which requires that the perturbation satisfy both local consistency conditions
\eqref{eq:Tau} and \eqref{eq:TauPrime} to attain the design order.
We observe that as before the inconsistent linearization  \eqref{linB2}  
results in first order  errors regardless of the value it is linearized around (red lines).
The linearization  \eqref{linB1}  results in second order 
errors regardless of the value it is linearized around (blue lines).
Finally, for the Taylor series linearization \eqref{linB3} we see third order when 
$\bar{y} = u^{(i-1)}$ (green dotted line) and, as expected, fourth order 
when $\bar{y} = u_n$ (green dashed line). Once again this performs exactly 
as the Newton iteration would for the unperturbed base method (black line).

Finally, at the bottom right we see the results from the B6s5p5m method \eqref{eq:B6s5p5m} 
which requires that the perturbation satisfy both local consistency conditions
\eqref{eq:Tau} and \eqref{eq:TauPrime}.
The inconsistent linearization  \eqref{linB2}  results in first order 
errors regardless of the value it is linearized around (red lines).
However, for larger $\dt$ the linearization around $\bar{y} = u_n$ (dashed line) 
is more accurate than the linearization around $\bar{y} =u^{(i-1)}$ (dotted line).
The linearization  \eqref{linB1}  results in second order 
errors regardless of the value it is linearized around (blue lines), though 
the linearization around $\bar{y} = u_n$ (dashed line) is more accurate than
the linearization around $\bar{y} =u^{(i-1)}$ (dotted line).
Finally, for the Taylor series linearization \eqref{linB3} we see third order when 
$\bar{y} = u^{(i-1)}$ (green dotted line) and, as expected, fifth order 
when $\bar{y} = u_n$ (green dashed line). Once again this performs exactly 
as the Newton iteration would for the unperturbed base method (black line).

We note that all the novel perturbed fourth and fifth order DIRK methods were tested for stability by 
increasing the number of spatial points $N_x$ while retaining a large $\dt$.
For the Taylor series linearization \eqref{linB3} we were not able to observe an unstable 
solution, even for the fifth order method where the underlying method is not B-stable.
We went up to a CFL value $\frac{\dt}{\dx} = 38$ where $N_x = 2400$ and $\dt =0.1$ and we were still able
to get an accurate and stable time-evolution.

 
\subsection{Shallow water equations}


Consider the scalar shallow water equations:
\begin{align}\label{sw1}
\eta_t + (\eta u)_x = 0, \; \; \; \; 
(\eta u)_t + \left( \eta u^2 + \frac{1}{2}\eta^2 \right)_x = 0,
\end{align}
for $x \in [0, 2\pi)$, with initial conditions
$\eta(x, 0) = \frac{\sin(x)}{10} +1$, $u(x, 0) = 0$, and periodic boundary conditions. Here $\eta(x, t)$ denotes the height and $u(x, t)$ the velocity.  
Let $\mu = \eta u$ be the mass flux, then \eqref{sw1}  can be written as 
\begin{align*}
\eta_t + \mu_x = 0,  \; \; \; \; 
\mu_t  + \left( \frac{\mu^2}{\eta} + \frac{1}{2}\eta^2 \right)_x = 0.
\end{align*}
Hence, similarly to the previous example, we semi-discretize this system of equations using a Fourier spectral method differentiation matrix $D_x$, and the function $f(y)$ is given by 
\begin{equation*}
    y' = \begin{pmatrix} y_{\eta}' \\
                y'_{\mu} \end{pmatrix} = f(y) = -\begin{pmatrix}
                    D_xy_\mu\\D_x\left[\frac{y_\mu^2}{y_\eta} + \frac{1}{2}y_\eta^2\right]
                \end{pmatrix}.
 \end{equation*}
 Given $y_n=  \bar{y} = \begin{pmatrix}\bar{{y}}_\eta\\\bar{y}_\mu
\end{pmatrix}$, let
$
\bar{Y_\eta} = \text{diag}(\bar{y_\eta}) \quad \mbox{and} \quad \bar{Y}_{\mu/\eta} = \text{diag}(\bar{y}_\mu/\bar{y_\eta}).
$

\begin{figure}[htb]
    \centering
\includegraphics[width=0.46\textwidth]{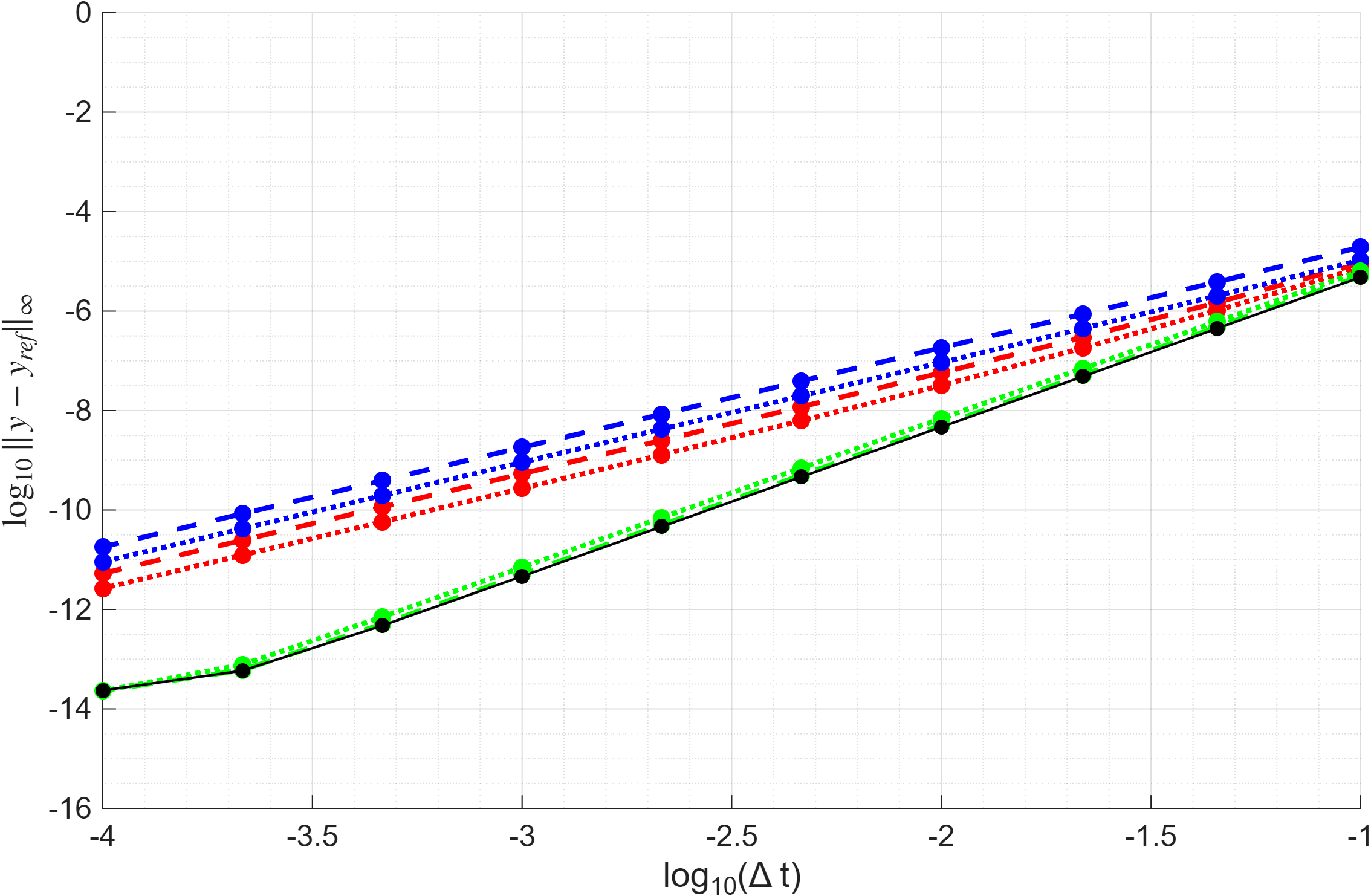}
\includegraphics[width=0.46\textwidth]{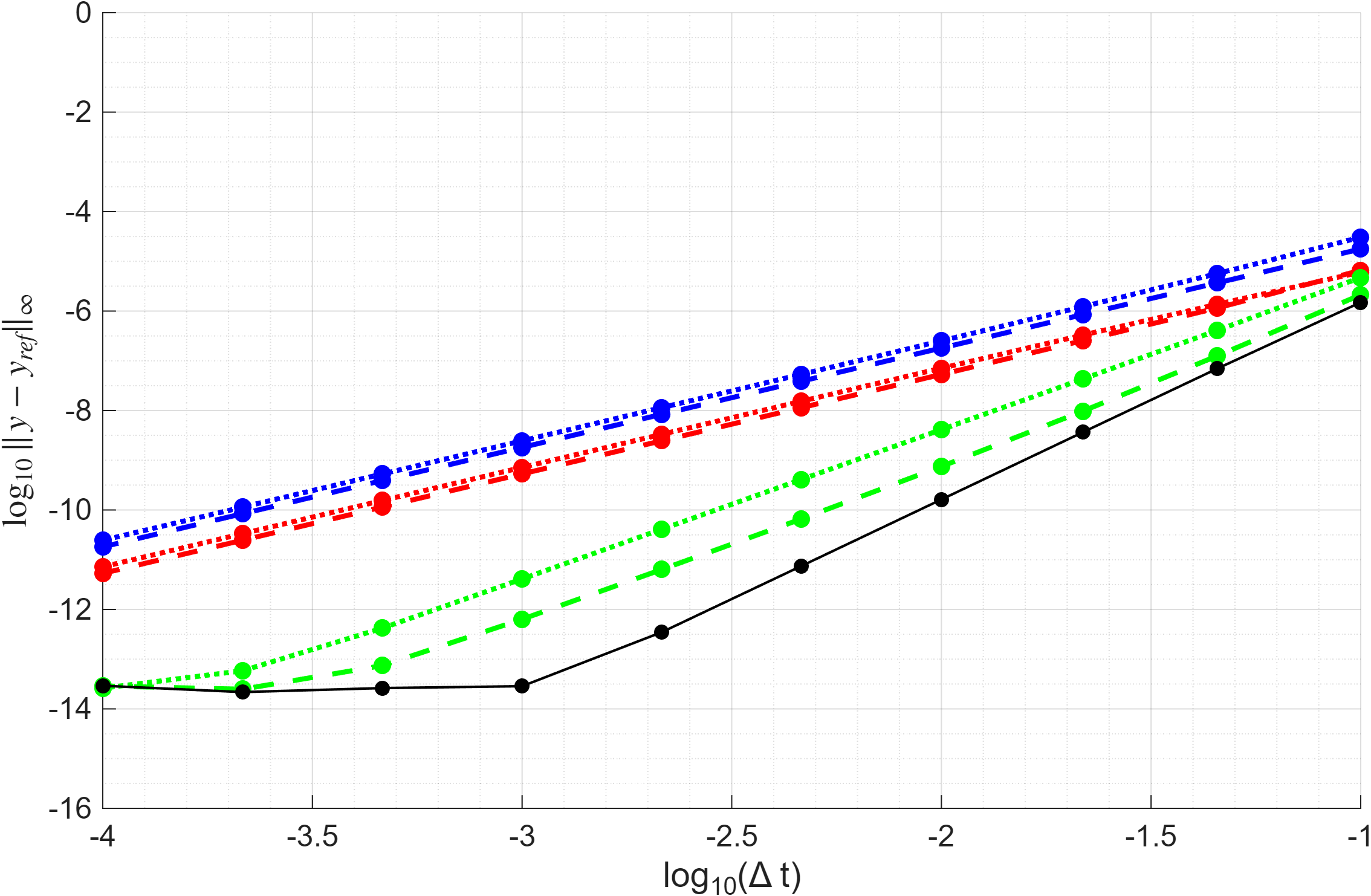}
    \\
\includegraphics[width=0.46\textwidth]{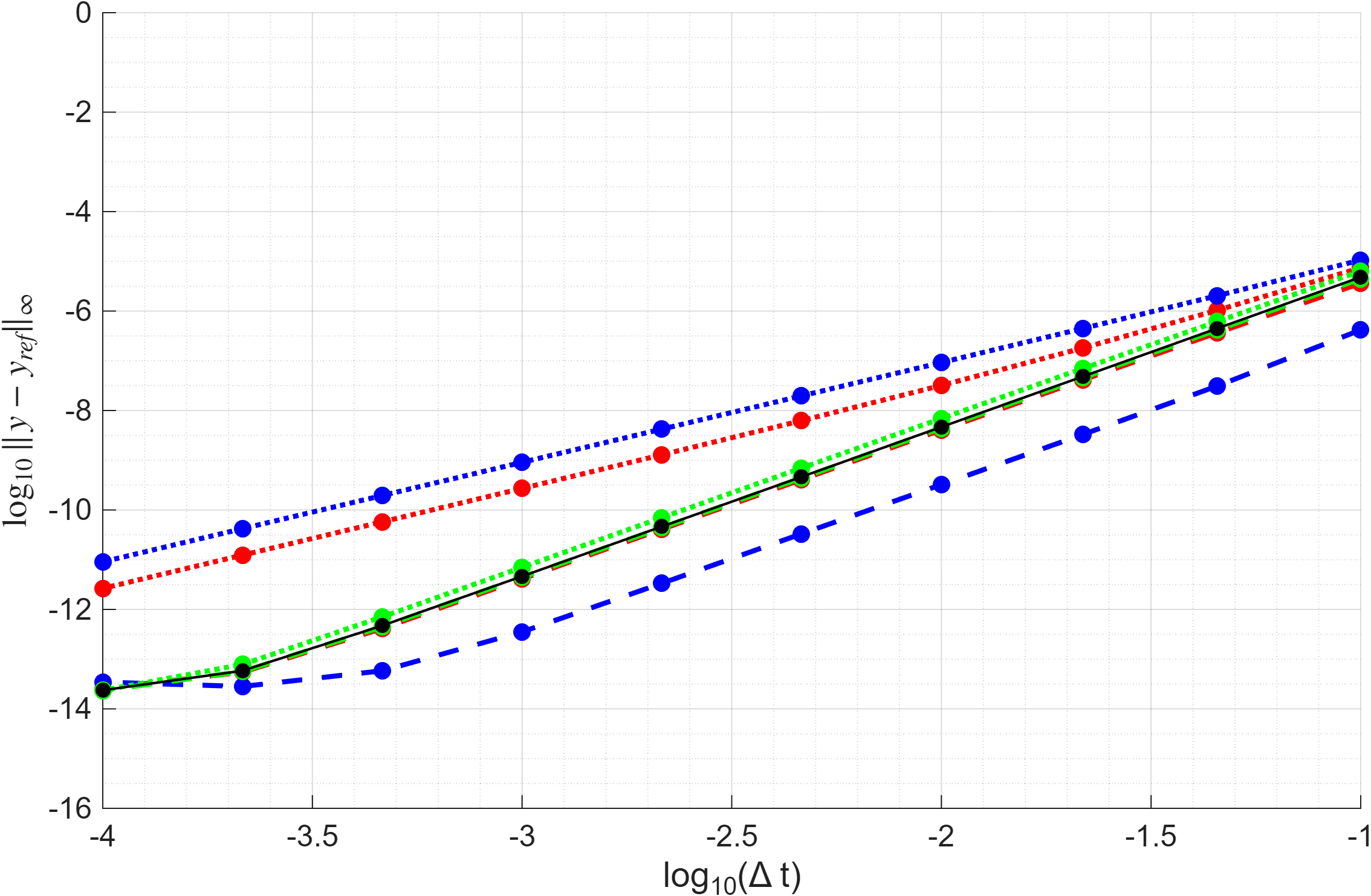}
\includegraphics[width=0.46\textwidth]{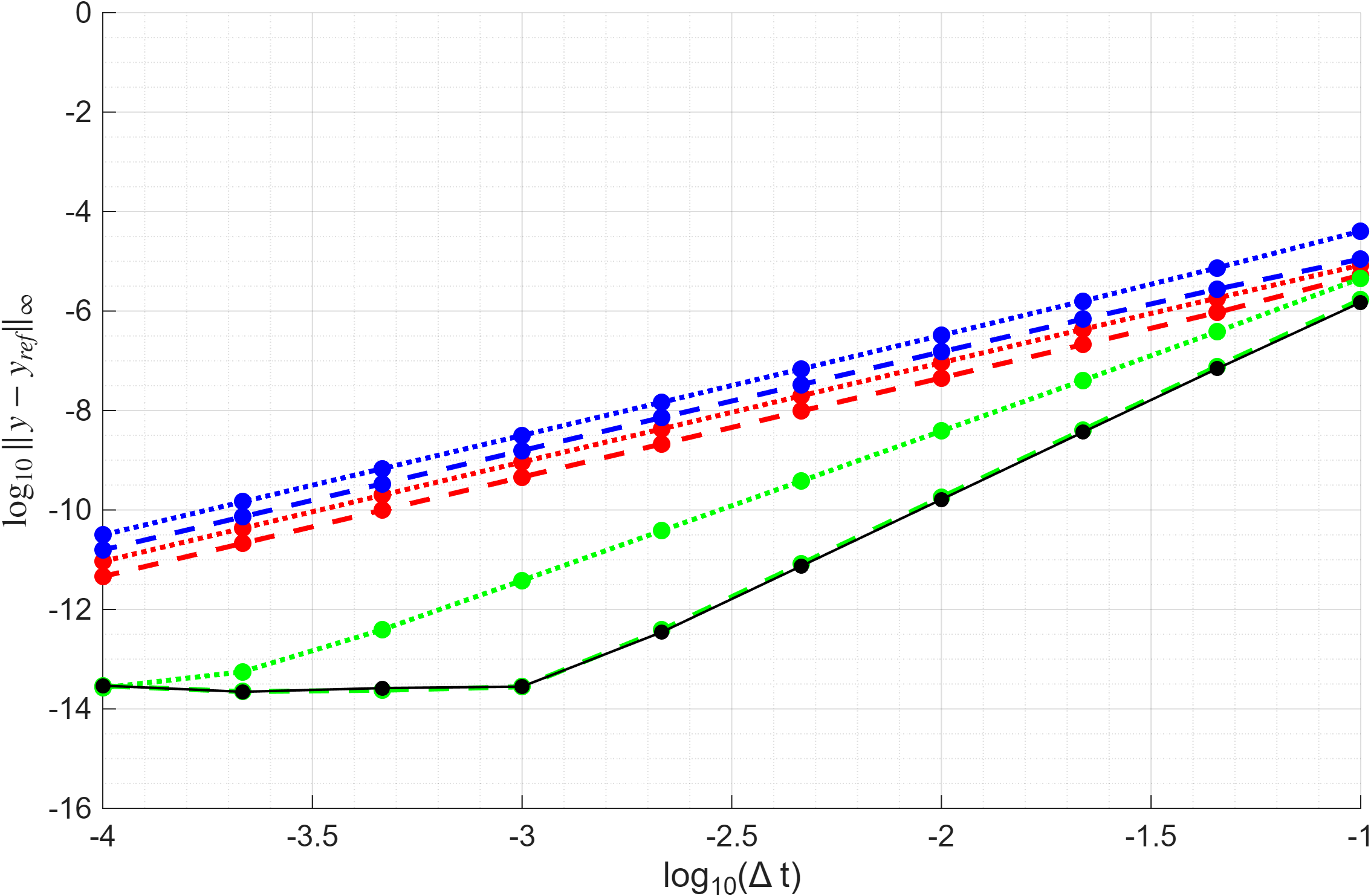} \\
    \begin{minipage}[c]{0.3\textwidth}  \hspace{0.1in}
        \includegraphics[width=0.85\textwidth]{Plots/PM_B_Lin/Master_Legend2.png}
    \end{minipage} \hfill
    \begin{minipage}[c]{0.65\textwidth} 
        \caption{Shallow water equations: final time errors for different linearizations 
        with
        using $N_x = 101$ spatial points. \\
        Top: diagonally perturbed methods D2s3p1m (left) and D3s4p1m method (right).
Bottom: new perturbed methods  A2s3p3m (left) and B3s4p4m method (right).
        \label{fig:ShallowWater}} 
    \end{minipage} \hfill 
\end{figure}

Similarly to the approach used previously for Burgers' equation, we compare the following three different linearizations around $\bar{y}$:
     \begin{eqnarray}
     \mbox{Linearize first: \; \; \; } 
     &    g_1(y) =  -\begin{pmatrix}
                    D_xy_\mu\\D_x\bar{Y}_{\mu/\eta}y_\mu + \frac{1}{2}D_x\bar{Y_\eta}y_\eta
                \end{pmatrix}. \label{linSW1} \\
    \mbox{Differentiate first:\;  } 
       &  g_2(y) =  -\begin{pmatrix}
                    D_xy_\mu\\2\bar{Y}_{\mu/\eta}D_xy_\mu + \left[  \bar{Y_\eta} - \left(\bar{Y}_{\mu/\eta}\right)^2\right]D_xy_\eta
                \end{pmatrix}, \label{linSW2} \\
       \mbox{Taylor series:} \; \;
       &  g_3(y)  = -\begin{pmatrix}
                    D_x\bar{y}_\mu\\D_x\left[\frac{\bar{y}_\mu^2}{\bar{y}_\eta} + \frac{1}{2}\bar{y}_\eta^2\right]
                \end{pmatrix} + f'(\bar{y})\begin{pmatrix}
                    y_\eta - \bar{y}_\eta\\ y_\mu - \bar{y_\mu}
                \end{pmatrix}, \label{linSW3} 
     \end{eqnarray}
     where 
     \begin{equation*}
         f'(\bar{y}) = \begin{pmatrix}
             {\bf 0} & -D_x\\D_x \left[\left(\bar{Y}_{\mu/\eta} \right)^2  - \bar{Y_\eta}  \right] 
             &-2D_x \bar{Y}_{\mu/\eta}
         \end{pmatrix}.
         \end{equation*}

In Figure  \ref{fig:ShallowWater} we compare the accuracy of the diagonally perturbed methods previously used
(top)  with our novel perturbed DIRK methods (bottom). 
Figure  \ref{fig:ShallowWater} (left) shows the final time errors using
        $N_x = 101$ spatial points, when evolved with the following
        diagonally perturbed method D2s3p1m  \eqref{pSDIRK3} (top left) 
        and the new perturbed method  A2s3p3m  \eqref{eq:A2s3p3m} (bottom left).
        Each of the linearizations \eqref{linSW1} (blue), \eqref{linSW2} (red), and \eqref{linSW3} (green)
        are performed for $\bar{y} = u_n$ (dashed lines) and $\bar{y} = u^{(i-1)}$ (dotted lines).
        
        We observe that for the Taylor series linearization \eqref{linSW3} the third order methods (left) 
        demonstrate third order methods, regardless of $\bar{y}$, 
        and for the linearization \eqref{linSW2} (red ) we see  second order regardless of $\bar{y}$.
        However, the difference is evident  for linearization \eqref{linSW1} (blue) which is second order regardless of $\bar{y}$ for the diagonally perturbed method D2s3p1m (top left), 
        and is second order for $\bar{y} = u^{(i-1)}$ (dotted lines) 
        for the new method A2s3p3m (bottom left), but is third order (and 
        results in even smaller errors than the linearization \eqref{linSW3} or exact evolution in black)
        when linearized around 
        $\bar{y} = u_n$ (dashed lines) and evolved with the new method A2s3p3m. This is because
        the diagonally perturbed method does not cause cancellations of the perturbations which 
        are a design feature of A2s3p3m when the linearization is consistent at  $\bar{y} = u_n$.

In Figure  \ref{fig:ShallowWater} on the right we see the diagonally perturbed fourth order method
D3s4p1m \eqref{pSDIRK4} top right), and the new perturbed method B3s4p4m \eqref{eq:B3s4p4m}
(bottom right).
We observe that for the diagonally perturbed D3s4p1m method  \ref{pSDIRK3} (top), 
the Taylor series linearization \eqref{linSW3} performs at third order, 
regardless of the choice of $\bar{y}$. 
However,  for the new perturbed method B3s4p4m \eqref{eq:B3s4p4m} (bottom) 
we see fourth order convergence when we  linearize around $\bar{y} = u_n$ (dashed lines).
The other linearizations all perform at second order, as expected.
This example shows the power of the new perturbed methods  A2s3p3m \eqref{eq:A2s3p3m} and
B3s4p4m \eqref{eq:B3s4p4m} when used with a perturbation that satisfies the correct 
local consistency conditions.


\subsection{A porous medium equation}


Our final example is the nonlinear equation 
 \begin{eqnarray}  \label{eq:PM}
 \mathcal{U}_t  = (\mathcal{U}^3)_{xx},
 \end{eqnarray}
 on the domain $x = [-\pi,\pi)$, 
 with initial condition $u(x,0)= \frac{1}{2} \cos(x)+  \frac{1}{2} $ 
 and periodic boundary conditions. Once again we use a spectral differentiation matrix
 for the spatial discretization, and evolve the resulting ODE system
  $ u' = f(u) =   D_{xx} u^{3}.$
 using the three diagonally perturbed  time-stepping methods
D1s2p1m \eqref{pIMR}, D2s3p1m \eqref{pSDIRK3}, and D3s4p1m \eqref{pSDIRK4},
and the novel methods in Section \ref{sec:NovelMethods}, 
to a final time $T_f = 0.5$.

 We can linearize $f$  in the following ways:
   \begin{eqnarray}
    \mbox{linearize first (blue):} \; \; 
   & f_{\epsilon}^1(\bar{y},u) =   D_{xx} \bar{Y}^2 u,  \; \;  \label{linPM1} \\
 \mbox{differentiate first (dark red):} \; \; 
  & f_{\epsilon}^{2,1}(\bar{y},u) =    3 D_{x} \left( \bar{Y}^2 D_x u \right),   \; \;  \label{linPM2a} \\
      \mbox{differentiate twice (orange): \; } 
   & f_{\epsilon}^{2,2}(\bar{y},u)  =   6 \bar{Y} \left( D_{x} \bar{y} \right) \odot 
   \left(D_x u \right)
   + 3\bar{Y}^2 D_{xx} u, \label{linPM2b} \\
    \mbox{Taylor series (green): \; \; } 
   & f_{\epsilon}^3(\bar{y},u) =   D_{xx} \bar{y}^{3}   + 3D_{xx}\bar{Y}^2(u-\bar{y}), 
 \label{linPM3}
   \end{eqnarray}

\begin{figure}[t!]
    \centering
    \begin{subfigure}[b]{\textwidth}
        \centering
        \includegraphics[width=0.49\textwidth]{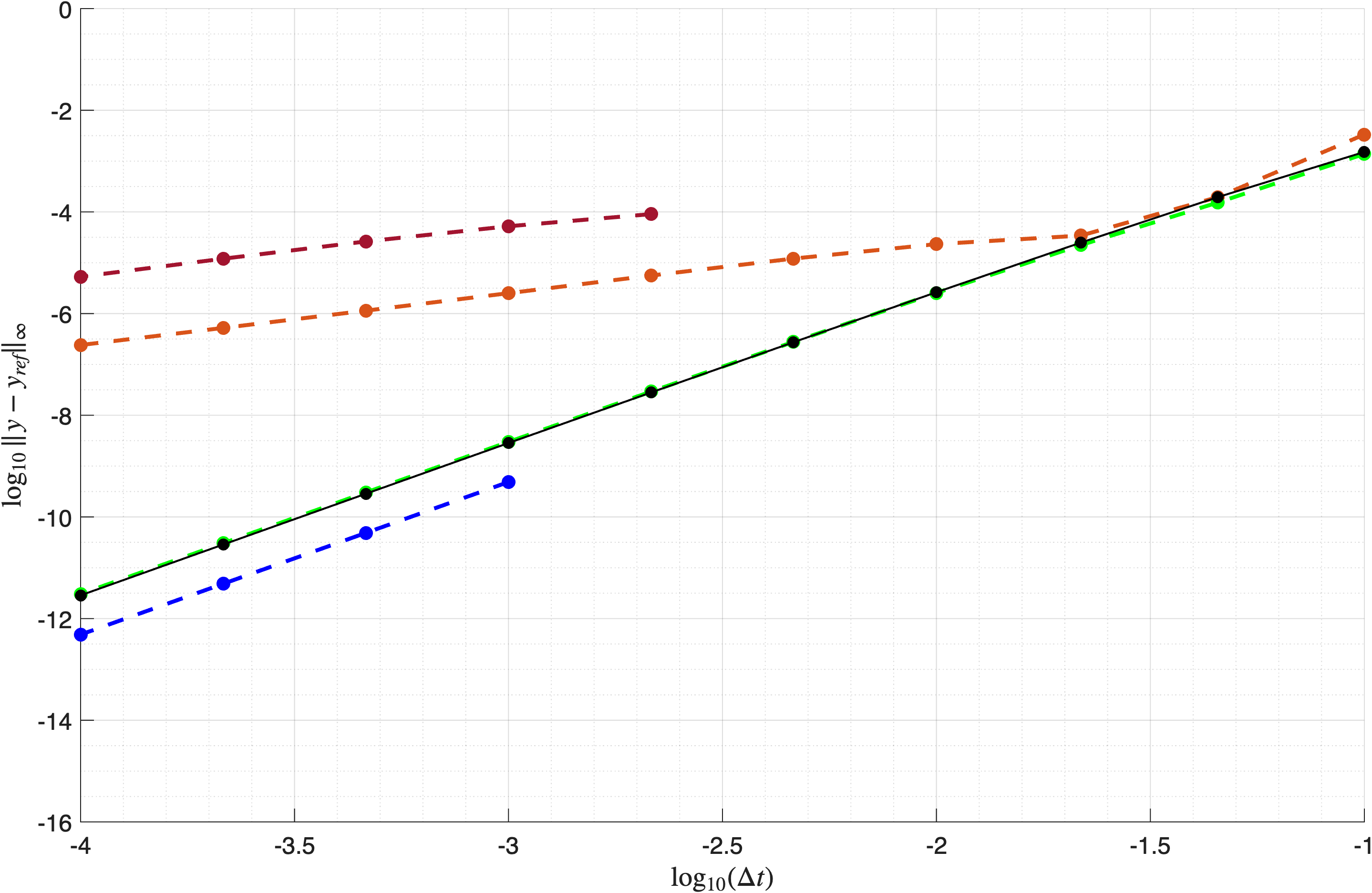}
        \includegraphics[width=0.49\textwidth]{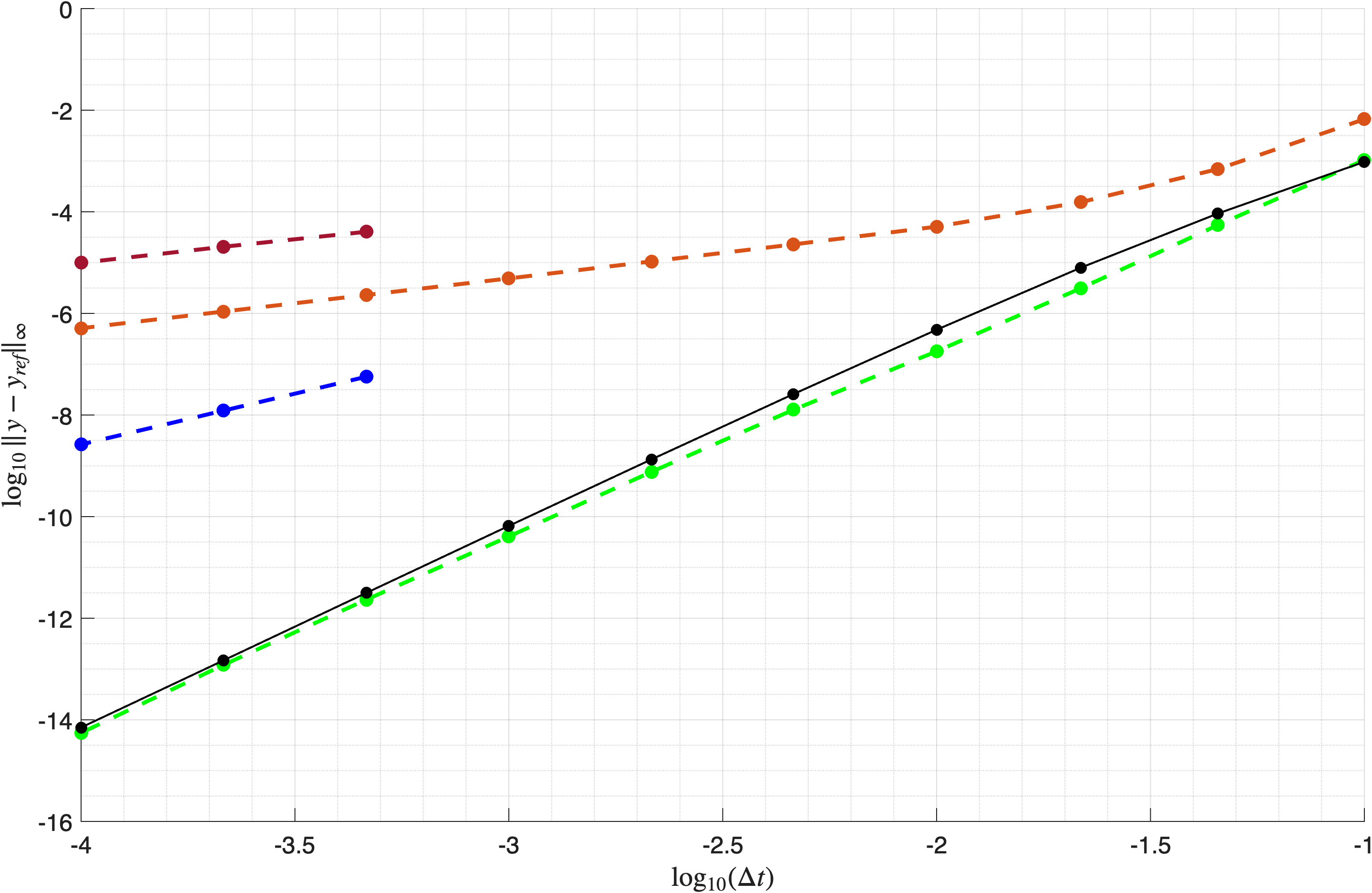}
    \end{subfigure} \\
    \begin{subfigure}[b]{\textwidth}
        \centering
        \includegraphics[width=0.49\textwidth]{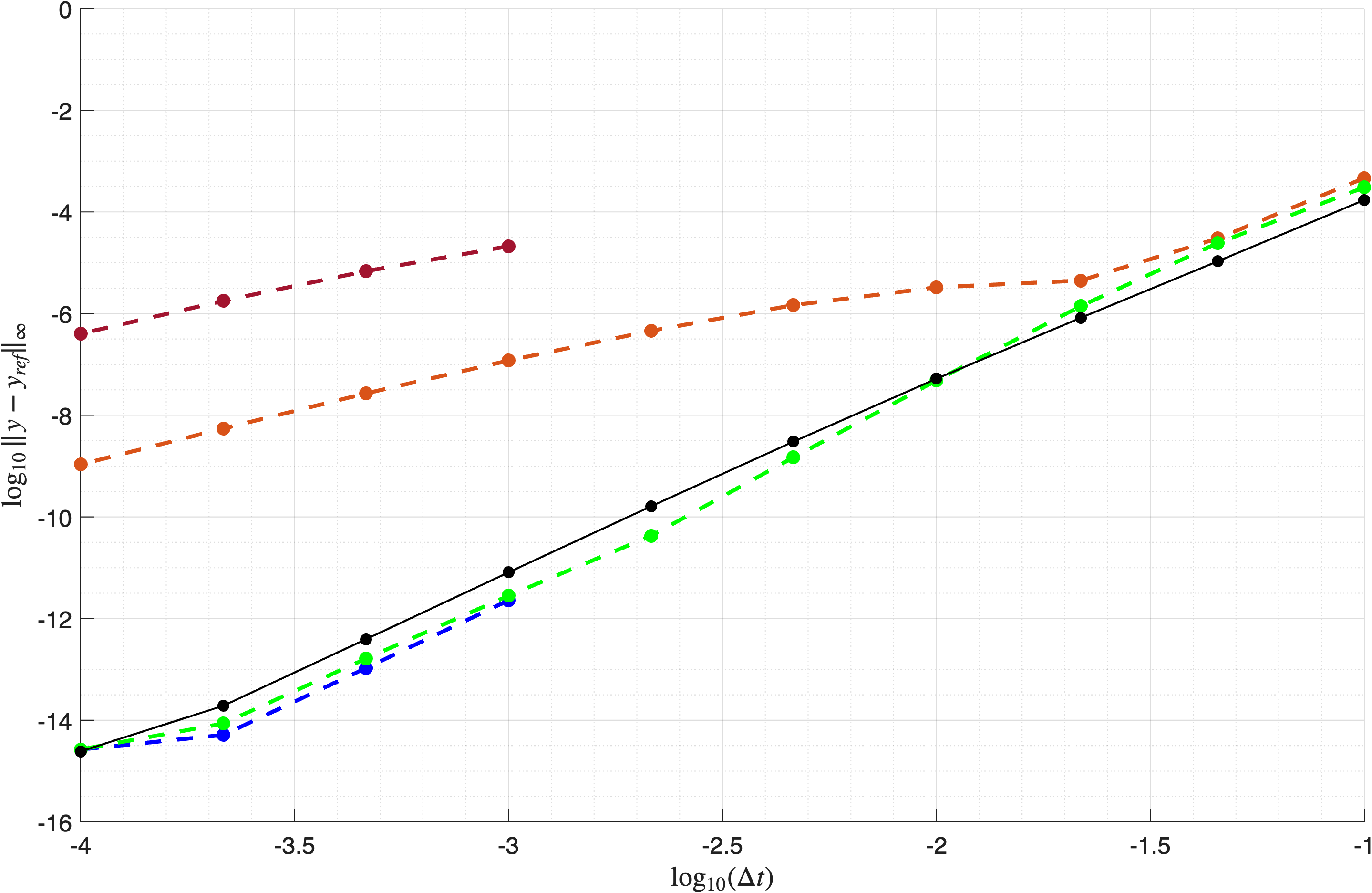}
        \includegraphics[width=0.49\textwidth]{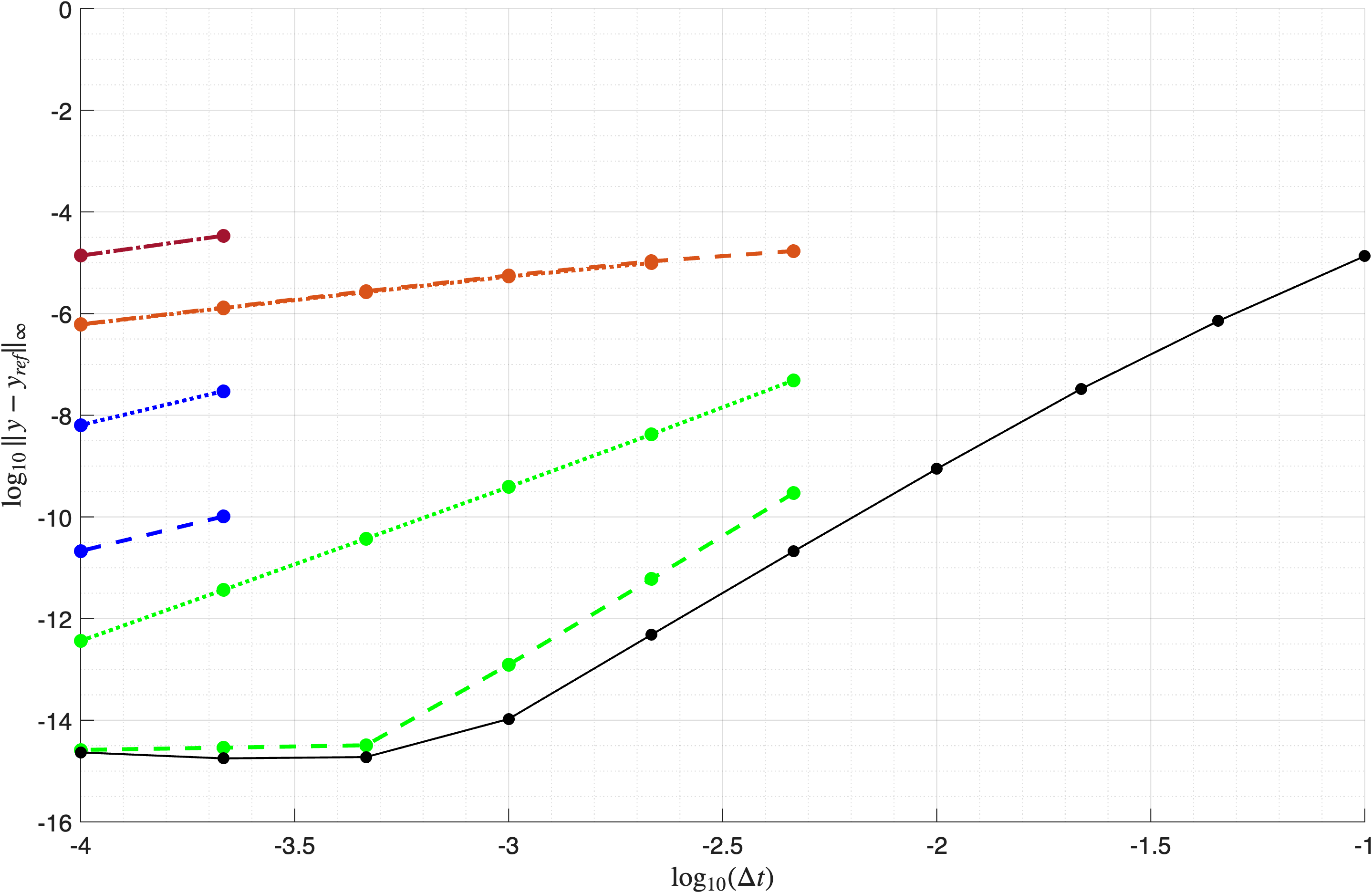}
    \end{subfigure}
    \begin{minipage}[c]{0.3\textwidth}  \hspace{0.1in}
        \includegraphics[width=0.85\textwidth]{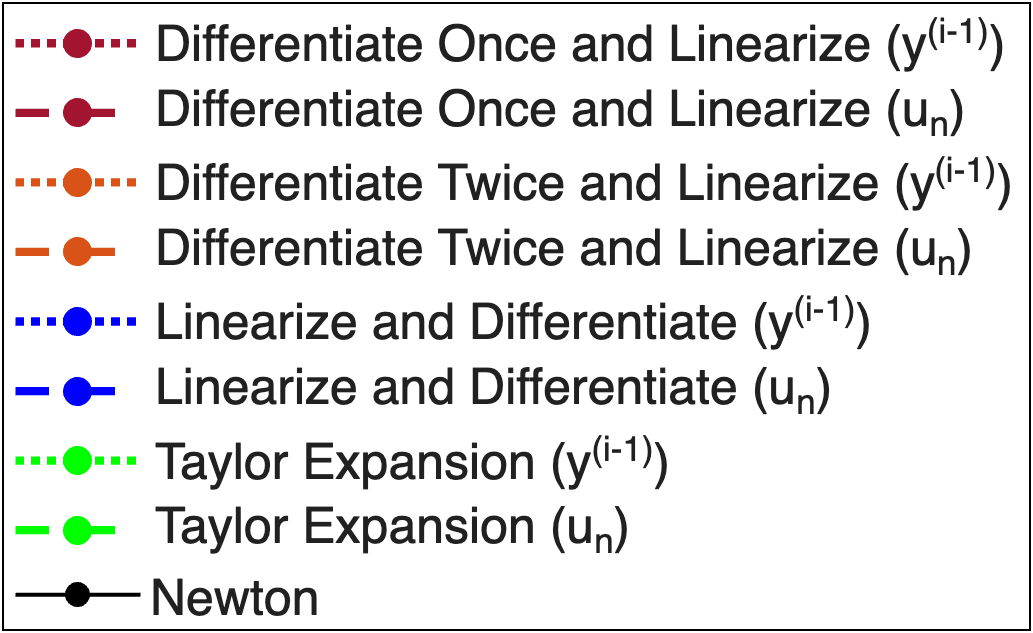}
    \end{minipage} \hfill
    \begin{minipage}[c]{0.65\textwidth} 
        \caption{Porous Medium equation: final time errors for different linearizations 
        with $N_x = 101$ spatial points. \\
        Top Left: A2s3p3m method \eqref{eq:A2s3p3m}, \\
        Top Right: B3s4p4m method \eqref{eq:B3s4p4m}. \\
        Bottom Left: A4s4p4m method \eqref{eq:A4s4p4m}, \\
        Bottom Right: B6s5p5m method  \eqref{eq:B6s5p5m}.
        \label{fig:PorousMedium_Perturbed}} 
    \end{minipage} \hfill 
\end{figure}

Figure \ref{fig:PorousMedium_Perturbed} shows the performance of the 
new methods in Section \ref{sec:NovelMethods}.
For all the methods we see that the linearizations \eqref{linPM2a} (red),  
\eqref{linPM2b} (orange) perform at first order, 
but that linearization  \eqref{linPM2a} is not as stable as  \eqref{linPM2b}.
The linearization \eqref{linPM1} (blue) performs at design order for the 'A' methods 
A2s3p3m and A4s4p4m, but loses accuracy and stability for the 'B' methods B3s4p4m and B6s5p5m.
The Taylor series linearization \eqref{linPM3} (green) performs at design order for all the methods 
except the fifth order method. However, the fifth order method is not stable for this problem 
for most values of $\dt$, not even for the unperturbed method.
Clearly, this example requires the B-stability property which this method does not satisfy.
Other than this, we see that the novel methods perform well in terms of accuracy and stability
for the Taylor series method.

\begin{figure}[htb]
\centering
\includegraphics[width=0.46\textwidth]{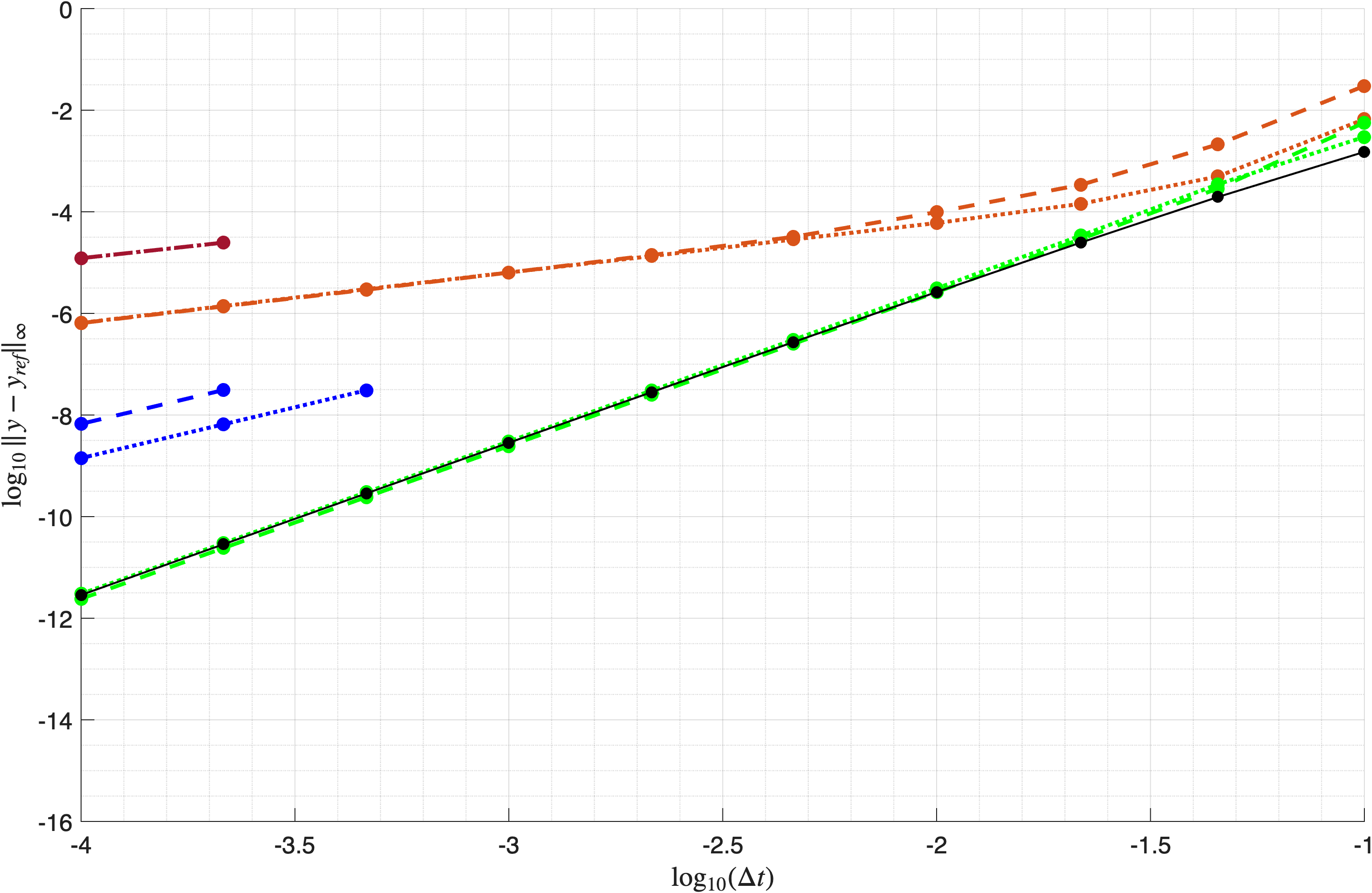}
\includegraphics[width=0.46\textwidth]{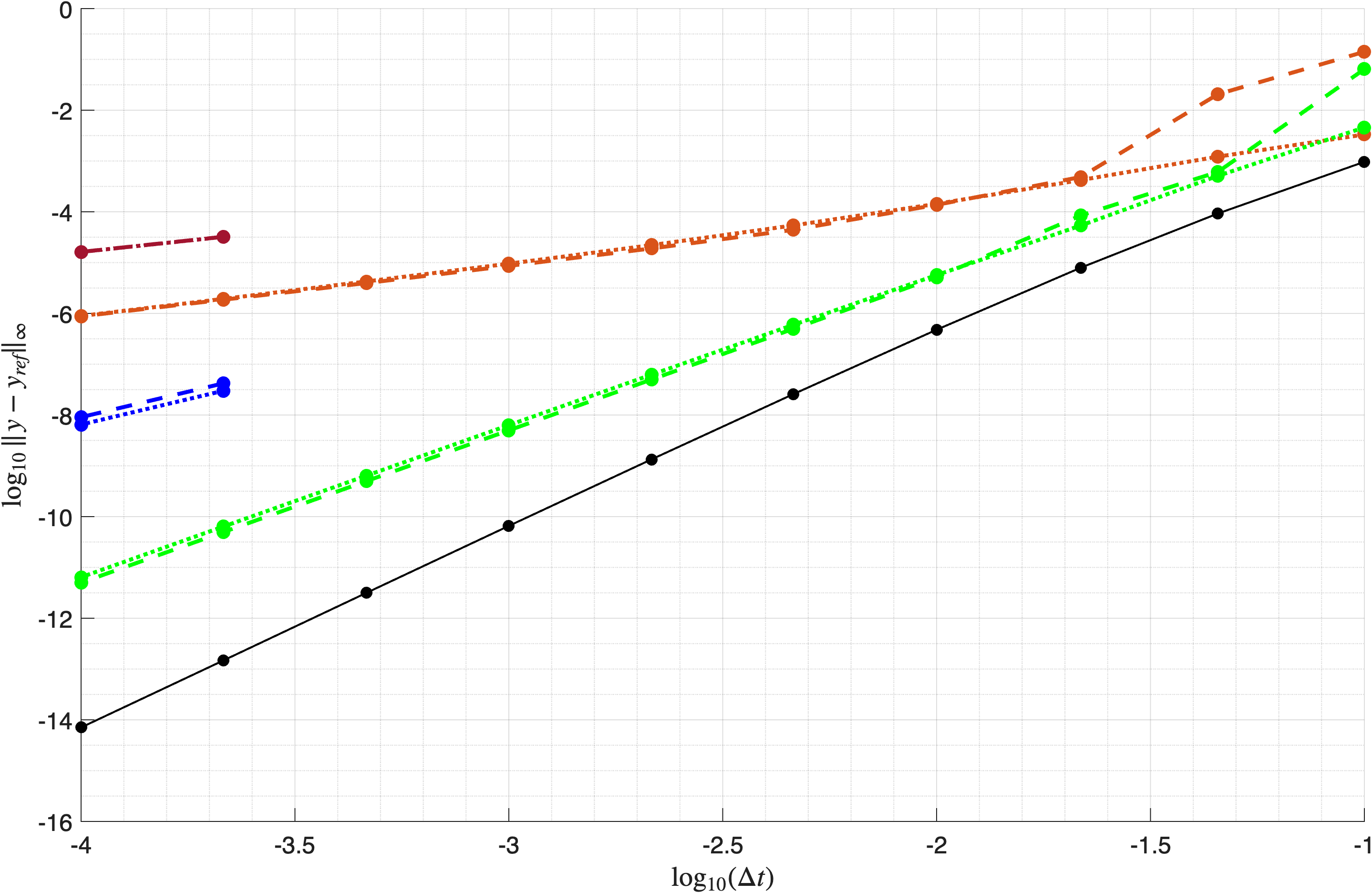}
\\
\includegraphics[width=0.46\textwidth]{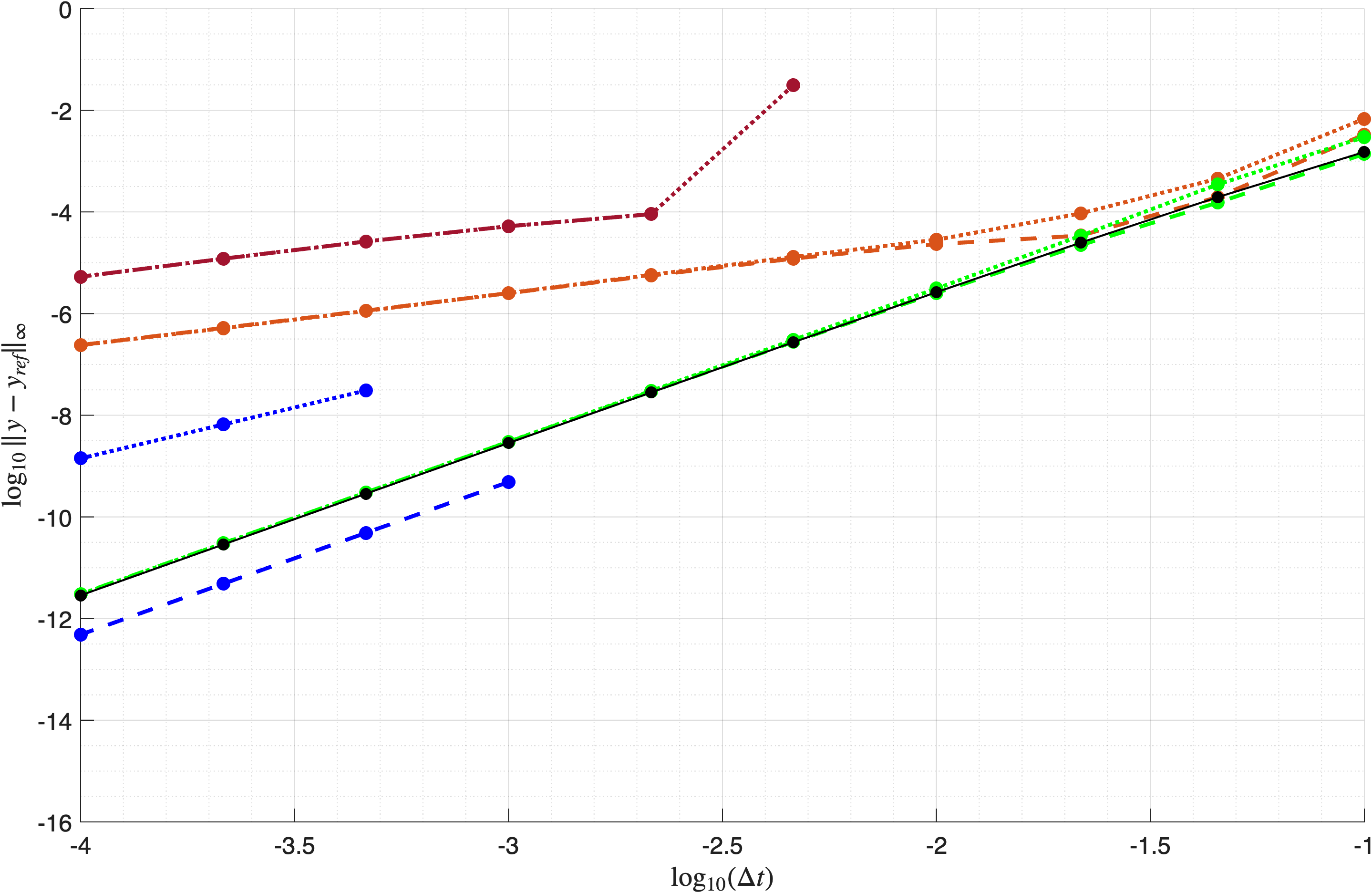}
\includegraphics[width=0.46\textwidth]{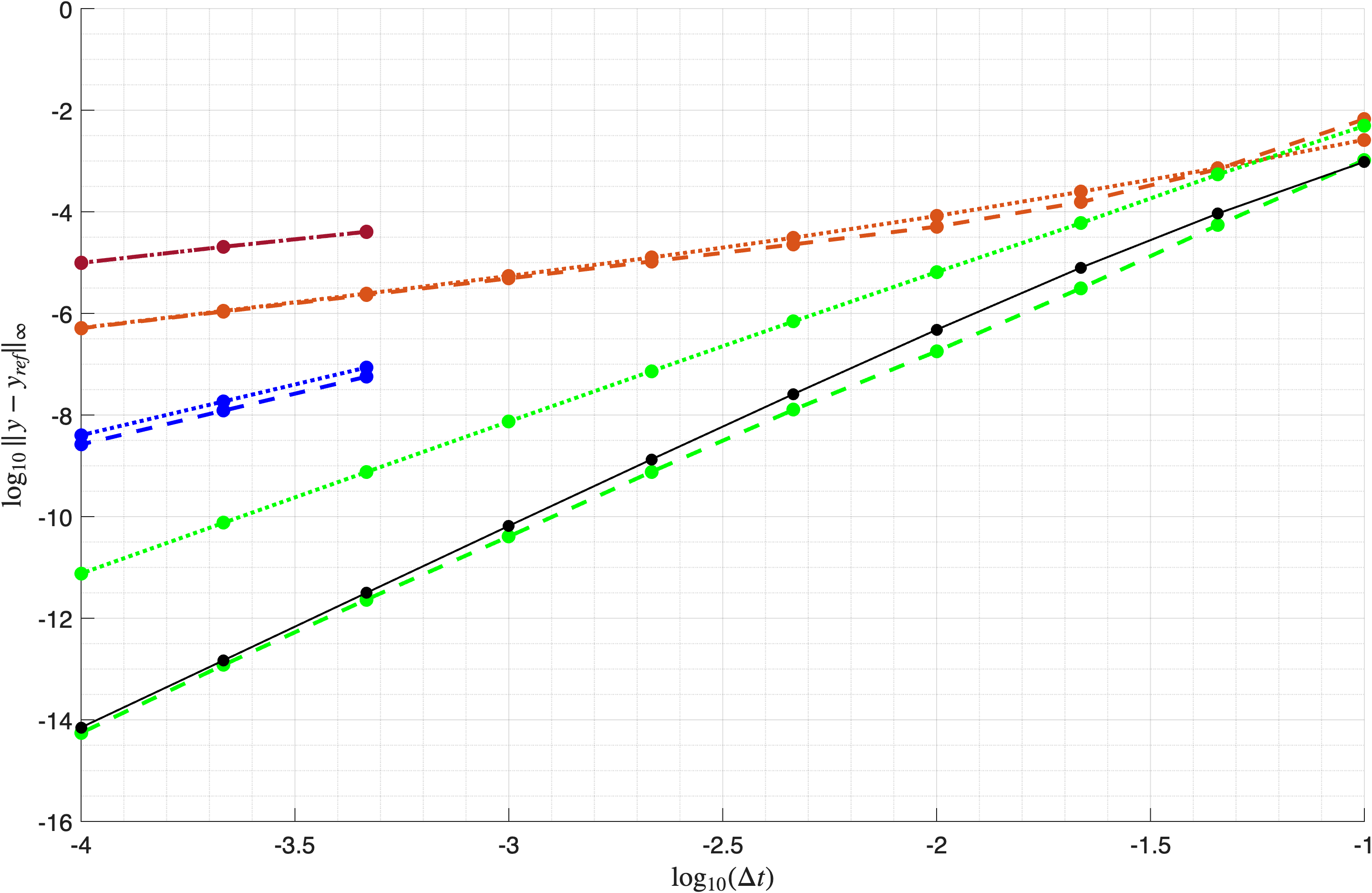}
\caption{Porous medium equation final time errors with  $N_x = 101$ spatial points. 
Top: diagonally perturbed methods D2s3p1m (left) and D3s4p1m method (right).
Bottom: new perturbed methods  A2s3p3m (left) and B3s4p4m method (right).
    \label{fig:PorousMedium_Standard}}
    \end{figure}

    Figure  \ref{fig:PorousMedium_Standard} (left) shows the final time errors using
        $N_x = 101$ spatial points, when evolved with the following
        diagonally perturbed method D2s3p1m  \eqref{pSDIRK3} (top left) 
        and the new perturbed method  A2s3p3m  \eqref{eq:A2s3p3m} (bottom left).
        Each of the linearizations  \eqref{linPM1} (blue),  \eqref{linPM2a} (red), 
        \eqref{linPM2b} (orange), and \eqref{linPM3} (green)
        are performed for $\bar{y} = u_n$ (dashed lines) and $\bar{y} = u^{(i-1)}$ (dotted lines).
        The linearizations  \eqref{linPM2b} (orange) and  \eqref{linPM2a} (dark red)
        behave similarly regardless of the value they are
        linearized around. However, we see that when we compare 
        the new perturbed method  A2s3p3m to the diagonally perturbed method D2s3p1m,
        the errors from the linearization  \eqref{linPM2b} are improved in terms of accuracy,  
        and  the \eqref{linPM2a} is far more stable.
        The loss of stability is consistent with the results in Section \ref{sec:Stability}, because in these
        cases we have a final time error growth of $O(\mathcal{K}_1 L \dt^2 T_f)$.
        If $\mathcal{K}_1 \approx L^2$ then $\mathcal{K}_1 L \dt^2$ will be very large if $\dt$ is not 
        sufficiently small.

        The linearization \eqref{linPM1} whether with $\bar{y} = u_n$ (dashed lines) and $\bar{y} = u^{(i-1)}$ (dotted lines)
        are not very stable for D2s3p1m (top left), and the $\bar{y} = u^{(i-1)}$ case is not very stable with 
        A2s3p3m (bottom left). However, matters improve dramatically for $\bar{y} = u_n$,
        in which case the errors are third order and even better than the unperturbed method 
        for small enough $\dt$.
        The Taylor series linearization \eqref{linPM3} performs at third order, as expected, in all the cases.
        
Figure  \ref{fig:PorousMedium_Standard} (right) shows the diagonally perturbed fourth order method
D3s4p1m \eqref{pSDIRK4} top right), and the new perturbed method B3s4p4m \eqref{eq:B3s4p4m} (bottom right).
The linearizations \eqref{linPM2a} and  \eqref{linPM2b} perform similarly to they did for the third order methods.
The linearization \eqref{linPM1} (blue) are slightly more stable for B3s4p4m, but not significantly.
However, for  the Taylor series linearization \eqref{linSW3}  we see that  in the  diagonally perturbed 
D3s4p1m method we get third order, regardless of the choice of $\bar{y}$. However, 
for the new perturbed method B3s4p4m (bottom) we see fourth order convergence when we 
linearize around $\bar{y} = u_n$ (dashed lines).
Once again, we see the potential  of the new perturbed methods  A2s3p3m \eqref{eq:A2s3p3m} and
B3s4p4m \eqref{eq:B3s4p4m} when used with a perturbation that satisfies the correct 
local consistency conditions.


\section{Conclusions}


The mixed-accuracy framework for Runge–Kutta methods introduced in \cite{Grant2022} 
offers a practical pathway for improving the efficiency of diagonally implicit Runge–Kutta (DIRK) 
methods by strategically replacing some of the  function evaluations $f$ with a more computationally 
efficient (but less accurate) variant $f_\varepsilon$.
These approximations introduce perturbations into the numerical method, which may be either smooth or 
nonsmooth in nature. While in the past we focused on nonsmooth perturbations,
particularly those arising from mixed-precision computations,  smooth perturbations are common and useful.
Examples of smooth perturbations include simplified models such as linearizations,
or under-resolved iterative solvers.

In this work we investigated the accuracy and stability properties of smoothly perturbed DIRK methods.
We develop a theory that relies on the smoothness of the perturbation error
$\tau = f_\varepsilon - f$, and use this theory to devise new methods that allow for efficient 
and accurate computations.  Our novel  perturbed DIRK methods are tailored to maintain high-order 
accuracy in the presence of smooth perturbations that satisfy additional local consistency conditions. 
The novel methods satisfy the usual accuracy conditions, so that they preserve the formal order, 
but also satisfy addition perturbation conditions that ensure  that the interaction between 
the perturbation and the time discretization does not degrade the overall solution quality.
Of particular interest in our investigation is a function $f_\varepsilon$ that arises from linearization
of $f$, and we show how the properties of the linearization match the local consistency conditions
we use to develop novel methods.

Finally, we present several  numerical experiments that further illustrate the impact of different types 
of smooth perturbations on both the newly developed SDIRK methods and existing methods from the literature. 
These results highlight the importance of accounting for perturbation structure in method design, 
and demonstrate that appropriately constructed methods can achieve improved robustness and efficiency 
in practical computations. Overall, this work provides both theoretical and computational evidence 
that mixed-accuracy strategies, when carefully analyzed and designed, can lead to high accuracy solutions.


\section*{\sc Author Contribution Statement} 


\noindent{\bf  John Driscoll} investigated numerous approaches and test cases.
He was primarily responsible for coding the numerical methods and replicating all the 
scalar results, and for checking the correctness and notation of the proofs.
JD reviewed the entire manuscript and  suggested many edits.

\noindent{\bf Sigal Gottlieb} was responsible for conceptualization of this project, 
With ZJG, She was primarily responsible for much of the stability analysis in Section \ref{sec:Stability}.
She worked with ZJG to understand and present the order conditions he developed in
Section \ref{sec:MADIRK}, suggested base methods to use for the development of novel
additive methods in Section \ref{sec:NovelMethods}, and 
determined numerical tests for Section \ref{sec:NumericalTests}.
SG was primarily responsible for writing and editing the manuscript.

\noindent{\bf Zachary J. Grant} was fully responsible for developing the smooth perturbation  framework for 
additive perturbed DIRK methods, and the related local consistency conditions and order conditions
in Section \ref{sec:MADIRK}. He was  responsible for finding all the novel methods
in Section \ref{sec:NovelMethods} and in designing the testing for Section \ref{sec:NumericalTests}.
He carefully proofread the entire paper,  and made numerous editorial suggestions 
to improve the presentation.

\noindent{\bf César Herrera}
was  involved in discussions on linearization and helped develop the 
understanding of these approaches. CH produced many numerical tests and graphs
for Section \ref{sec:NumericalTests}, and in particular is responsible for all
the work on the shallow water equations.
He read, commented, and edited the entire manuscript.

\noindent{\bf Tej Sai Kakumanu} was responsible for enhancing the understanding of
linearizations and suggesting approaches and numerical tests for 
Sections \ref{sec:NumericalTests}, and performing many of the numerical tests.
TSK   reviewed the entire manuscript and  suggested edits.

\noindent{\bf Monica Stephens} 
was responsible for the conceptualization of the problem and for the 
presentation of the work. She carefully reviewed and edited the entire manuscript for  
mathematical correctness and  made multiple editorial changes that contributed to the clarity of the 
manuscript.


\section*{\sc Funding Information}


The authors' research was supported in part by 
AFOSR Grant No. FA9550-23-1-0037,
NSF Grant No.  DMS-2309609,
and
DOE  Grant No. DE-SC0023164 Subaward RC114586.
SG acknowledges the support of  Mass Dartmouth’s Marine and Undersea Technology 
(MUST) Research Program funded by the ONR Grant No. N00014-20-1-2849. 
MS acknowledges support from the National Science Foundation PRIMES program under Grant No. DMS-2331890.
The authors acknowledge the Unity Cluster managed by the Research Computing 
\& Data team at the University of Massachusetts Amherst, and the UMassD shared 
cluster as part of the Unity cluster, supported by   AFOSR DURIP grant FA9550-22-1-0107.


\section*{Appendix} \addcontentsline{toc}{section}{Appendix}

\subsection*{A.1 Proofs of Section \ref{sec:Stability}} \label{app:proof}

In this subsection we generally assume that the ODE is scalar, however,
to remind us that this can easily be converted to a system we retain the notation 
$\left\| E_n \right\|$.

\smallskip

\begin{proof}  {\sc Lemma \ref{lem0}:}
The stability analysis follows from looking  at the inner product 
of the difference $E_{n} = z_{n} -y_{n}$ at each time-step:
\begin{eqnarray*}
\|E_{n+1}\|^2   &=&  \|E_n +
 \sum_{i=1}^s b_{i}  \psi^{(i)}  \|^2 = \| E_n\|^2 + 2  \sum_{i=1}^s b_{i} E_n^T \psi^{(i)}
 + \left( \vb \Psi ,\vb \Psi \right) \\
  &=&  \| E_n\|^2  + \left( \vb \Psi ,\vb \Psi \right)  
  + 2 \sum_{i=1}^s b_{i} (\psi^{(i)})^T \left( E^{(i)}  -  \sum_{j=1}^{i} a_{ij} \psi^{(j)} 
    + \dt \sum_{j=1}^i a^\varepsilon_{ij} \tau(y^{(j)})
    \right) \\
    &=&  \| E_n\|^2  - \left(  \Psi ,M \Psi \right) 
+ 2 \sum_{i=1}^s b_{i}  \left( \psi^{(i)} , E^{(i)} \right) 
 + 2 \dt \sum_{i=1}^s \sum_{j=1}^i b_{i}  a^\varepsilon_{ij}
 \left( \psi^{(i)} ,  \tau(y^{(j)}) \right)  \\
    &\leq &  \| E_n\|^2   + 2 \dt \; 
 \left\|  \vb^T  \Aep  \Psi   \tau(\vy)   \right\|
\end{eqnarray*}
The inequality follows from the fact that   $M$ is semi-positive definite by assumption, 
so that $( \Psi, M \Psi)  \geq 0$, and the positivity of $b_i$ combined with
the contractivity of $f$ that ensures that $  \left( \psi^{(i)} , E^{(i)} \right) \leq 0$.
\end{proof}

\begin{proof} {\sc Lemma \ref{lem1}:}
The positivity of  $a_{ii}$
and the contractivity of $f$ give
\begin{eqnarray*}
\left\| E^{(i)} \right\|^2 & \leq & \left(E^{(i)},E^{(i)}\right) -2a_{ii} \left( E^{(i)}, \psi^{(i)}\right)
+ a_{ii}^2\left( \psi^{(i)}, \psi^{(i)}\right) = \left\| E^{(i)} - a_{ii} \psi^{(i)} \right\|^2.
\end{eqnarray*}
If we let $\mAh$ be a diagonal matrix with  diagonal elements $a_{ii}$:
\begin{eqnarray*}
\|  \vE  \| & \leq  & \|  \vE -   \mAh \Psi \|
=  \left\| E_n \ve    +  \mA  \Psi - \dt  \Aep \tau(\vy) -   \mAh \Psi  \right\| 
\end{eqnarray*}
We now replace $  \Psi = \mA^{-1} \left( \vE -    E_n \ve + \dt \Aep \tau(\vy) \right) $
to get
\begin{eqnarray*}
\|  \vE  \| & \leq  & \left\|  E_n \ve +  (\mA- \mAh) \mA^{-1} \left(\vE -   E_n \ve 
+ \dt \Aep \tau(\vy) \right)   - \dt  \Aep \tau(\vy)  \right\|\\
& \leq  & \left\|   \mAh \mA^{-1}   E_n \ve     \right\| 
+ \left| \mI - \mAh  \mA^{-1} \right|  \left\|  \vE    \right\| 
+\dt  \left\|   \mAh \mA^{-1}  \Aep    \tau(\vy)    \right\|  
\end{eqnarray*}
Rearranging, we obtain the result. 
\end{proof}

\begin{proof} {\sc Theorem \ref{thm:stability}:}
We first bound $\left\|  \Psi \right\|$
\[ \left\|  \Psi \right\|   \leq   \dt L \left\| \vE \right\| 
\leq \dt L C_1  \|E_n\|  + \dt^2 L C_2  \left\|   \tau(\vy)    \right\|
\]
Now we want to bound the terms  $\left\|  \tau(\vy)   \right\|$. 
If $\tau(y_n) =0 $ but we may have  $\tau_u(y_n) \neq 0 $ then we  expand
\begin{eqnarray*}
\left\|  \tau(\vy )   \right\|  & \leq  & 
\left\| \tau_u(\zeta) \right\| \left\| \vy - y_n \right\| 
\leq M_1 \dt
\end{eqnarray*}
for some $M_1$.
If, we know that $\tau(y_n) =\tau_u(y_n) = 0 $ then the expansion becomes
\begin{eqnarray*}
\left\|  \tau(\vy)   \right\| & = & \left\|  \tau(y_n) + \tau_u(y_n) (\vy - y_n)  
+ \frac{1}{2} \tau_{uu}(\eta) (\vy - y_n)^2 \right\| 
 \leq M_2 \dt^2,
\end{eqnarray*}
for some $M_2$.

\begin{itemize}
\item For the case where $\tau_u(y_n) \neq 0$:
\begin{eqnarray*}
\|E_{n+1}\|^2   &\leq &  \| E_n\|^2   + 2 \dt \; \vb \left| \Aep \right| \ve
 \left\|  \Psi  \right\|   \left\|  \tau(\vy)   \right\| \\
  &\leq &  \| E_n\|^2   + 2 M_1 \dt^2 \vb \left| \Aep \right| \ve
  \left( 
   L C_1   \left\| E_n  \right\|  + \dt^2 L M_1 C_2 
  \right) \\
 & \leq &  \| E_n\|^2   +  2 \mathcal{K}_1 L \dt^3  \left\| E_n  \right\|
  + \tilde{K} L^2\dt^4 .
\end{eqnarray*}
This becomes
$ \|E_{n+1}\|^2   \leq  \left(\| E_n\|   + \mathcal{K}_1 L \dt^3\right)^2  
  + L^2 \dt^4 \left(\tilde{K}  - \mathcal{K}_1^2 \dt^2 \right). $
  For sufficiently large $\dt$ (such that $\tilde{K}  \leq  \mathcal{K}_1^2 \dt^2  $)
  we have $ \|E_{n+1}\|   \leq  \| E_n\|   + \mathcal{K}_1 L \dt^3 .$
\item For the case where $\tau_u(y_n) = 0$ we have 
\begin{eqnarray*}
\|E_{n+1}\|^2   &\leq &  \| E_n\|^2   + 2 \dt \;  \ \left| \vb^T \Aep \right| 
 \left\|  \Psi  \right\|   \left\|  \tau(\vy)   \right\| \\
 &\leq &  \| E_n\|^2   + 2 M_2 \dt^4 \; \vb \left| \Aep \right| \ve 
 \left(L C_1   \left\| E_n  \right\|  + \dt L C_2 \left\|   \tau(\vy)    \right\|
 \right) \\
 &\leq &  \| E_n\|^2   + 2 \mathcal{K}_2 L \dt^4  \left\| E_n  \right\|
 +  \hat{K} L^2 \dt^7 
 \end{eqnarray*}
\end{itemize}
Once again, by completing the square
\begin{eqnarray*}
\|E_{n+1}\|^2   &\leq & \| E_n\|^2   + 2 \mathcal{K}_2 L \dt^4  \left\| E_n  \right\| + 
\mathcal{K}_2^2 L^2 \dt^8 
 +  L^2  \dt^7 \left(\hat{K}   - \mathcal{K}_2^2  \dt \right),
 \end{eqnarray*}
 so that if $\dt$ is large enough so that $\hat{K}  \leq  \mathcal{K}_2^2  \dt $ we have
$ \|E_{n+1}\|   \leq    \| E_n\| +  \mathcal{K}_2 L \dt^4 .$
  \end{proof}

\newpage

  \subsection*{A.2 On Linearizations} \label{app:linearizations}

  Given  nonlinear PDEs of the form
\[\cU_t =  (\cU^m)_x \; \; \; \mbox{or}  \; \; \; \cU_t =  (\cU^m)_{xx}.\]
we can discretize in space using a differentiation matrix, to get an ODE of the form
\[u_t = f(u) = D_x(u^m) \; \; \; \mbox{or}  \; \; \; u_t = f(u) = D_{xx}(u^m) .\]
\begin{itemize}
    \item {\bf Linearization 1:} 
    A simple approach to linearization is  to replace $u^{m-1}$ with  $u_n^{m-1}$
    so that 
   \[f_\varepsilon(u_n,u) = D_x(U_n^{m-1} u ) \; \; \; \mbox{or}  \; \; \; 
   f_\varepsilon(u_n,u) = D_{xx}(U_n^{m-1} u ),\] 
   (the multiplication of two vectors here is understood componentwise as the Hadamard product).
   
   We can easily see that in both cases $f_\varepsilon(u_n,u_n) = f(u_n)$ so that condition \eqref{eq:Tau} is
   satisfied.
   However, we can also see that in both cases condition \eqref{eq:TauPrime} is not satisfied:
   \[ \frac{\partial f_\varepsilon}{\partial u} = D_x u_n^{m-1} 
   \; \; \; \mbox{or}  \; \; \;
   \frac{\partial f_\varepsilon}{\partial u} = D_{xx} u_n^{m-1},
   \]
   so that 
   \[ \tau_u (u_n) = \frac{\partial f_\varepsilon}{\partial u}(u_n) - 
   \frac{\partial f}{\partial u}(u_n) = D_x u_n^{m-1}  - D_x (m u_n^{m-1}) \neq 0,
   \]
   or 
   \[ \tau_u (u_n) = \frac{\partial f_\varepsilon}{\partial u}(u_n) - 
   \frac{\partial f}{\partial u}(u_n) = D_{xx} u_n^{m-1}  - D_{xx} (m u_n^{m-1}) \neq 0.
   \]
   \item {\bf Linearization 2:}  We can also rewrite the PDE in a different form and then linearize based on that form.
   For example, 
   \[ \cU_t =  (\cU^m)_x \; \; \; \Rightarrow \; \; \; \cU_t = m \cU^{m-1} \cU_x \; \; \; \Rightarrow \; \; \;
   u_t = f(u) = m u^{m-1}  D_x u  ,\]
   so that we can linearize 
   $ f_\varepsilon(u_n,u) = m U_n^{m-1}   D_x u.$

   Now, if we think of $f = D_x(u^m)$ then clearly we have a mismatch between $f$ and $f_\varepsilon$,
   because
   \[ \tau(u_n) = D_x(u_n^m) - m U_n^{m-1}   D_x u_n \neq 0. \]
   Recall that the differentiation operator does not satisfy the property
   $ D_x u^m = m u^{m-1} D_x u,$ but that this difference converges with $N_x$. This behavior was observed 
   in the numerical results.
   
   The second PDE $\cU_t =  (\cU^m)_{xx}$ can be modified, 
   and therefore linearized, in several ways:
   \begin{enumerate}
       \item If we differentiate once we get 
       \[ \cU_t = \frac{d}{dx} \left( m \cU^{m-1} \cU_x\right) \; \; \; \Rightarrow 
       u_t = f(u) = D_x \left( m u^{m-1} D_x u \right) 
       \] 
       so that we can linearize 
   $ f_\varepsilon(u_n,u) = D_x \left( m U_n^{m-1} D_x u \right) $.
       \item We can take this one step further and differentiate out the first term
       \[ \cU_t = m (m-1) \cU^{m-2} \left( \cU_x \right)^2 
       + m \cU^{m-1}  \cU_{xx}   \]
       which becomes the ODE 
       \[ u_t = f(u) = m (m-1) u^{m-2} \left( D_x u \right)^2 + m u^{m-1}  D_{xx} u.\]
       This can be linearized around $u_n$ 
       \[ f_\varepsilon(u_n,u) =  m (m-1) U_n^{m-2} \left( \left( D_x u_n \right) \odot \left(D_x u \right) \right)
       + m U_n^{m-1}  D_{xx} u. \]
   \end{enumerate}
    In all these cases, we may expect $\tau(u_n) = 0$, or at least that
    $ \tau(u_n) \rightarrow 0 $ as $N_x$ increases, but we do not expect $\tau_u(u_n) = 0$.
    \item {\bf Linearization 3:}  Finally, we consider a linearization based on a 
    Taylor expansion around $u_n$.
    Any $f(u) $ will be linearized as
    \[ f_\varepsilon(u_n, u)  = f(u_n) + f'(u_n) ( u - u_n) .\]
    In this case $ \tau(u) = f(u_n) + f'(u_n) ( u - u_n) - f(u) $ so that 
$ \tau_u(u) = f'(u_n) - f'(u).$ Clearly then  $ \tau_u(u_n) = f'(u_n) - f(u_n) =0 .$
Thus, this linearization satisfies both \eqref{eq:Tau} and \eqref{eq:TauPrime}, and so
is preferred as it simplifies the order conditions in Table \ref{tab:OC}. In fact,
this is a more accurate linearization and we will see that it generally performs 
better than the other alternatives.

For the PDEs considered here, the Taylor series linearization takes the form:
\[ f_\varepsilon(u_n,u) =  D_x(u_n^m) + D_x m U_n^{m-1} (u-u_n), \]
and 
\[ f_\varepsilon(u_n,u) =  D_{xx}(u_n^m) + D_{xx} m U_n^{m-1} (u-u_n), \]
where $U_n = diag(u_n)$.

\end{itemize}

  \subsection*{A.3 Fifth order method} \label{app:5thO}

  This six stage fifth order method is given by the coefficient $\tilde{\mA}$ and $\mA^\varepsilon$ where $\tilde{\mA} = \mA - \mA^\varepsilon$ and 
  \[ 
\begin{array}{ll}
 a_{21} = 0.27805384113645232493158618529853  &
 a_{22} = 0.27805384113645232493158618493986   \\
 a_{31} = 0.02563926406019955725750334911617  &
 a_{32} = -0.067907140638319686154545423375094  \\
a_{33} = 0.27805384113645232493158618493986 &
a_{41} = 0.46288174618101471080052219497850 \\ 
a_{42} = 0.15087176590423228508246469950504 & 
a_{43} = -0.27175112661133156378059421421552 \\
a_{44} = 0.27805384113645232493158618493986 &
a_{51} =  0.6654544604241820026463783076237 \\  
a_{52} = 7.0986011802448438993227520426319 &  
a_{53 } = -1.7957575925396292876887601404390 \\
a_{54} =  -5.0184428584790110818489484844777 &   
a_{55} =0.27805384113645232493158618493986 \\
a_{61} =  0.10731715308473725999947312385777 & 
a_{62} = 0.87998595709151287098426474073169 \\ 
a_{63} =0.19747904757470600132499354237387 &
a_{64} = -0.41841860591874017976596344143339 \\
a_{65} = -0.044417392968668277474354150469769 \; \; \; & 
a_{66} = 0.27805384113645232493158618493986 \\ 
 \end{array}  \]

\[ 
\begin{array}{ll} 
b_1 = 0.10731715308473725999947312385777 &
b_2 =    0.87998595709151287098426474073169\\
b_3 =    0.19747904757470600132499354237387&
b_4 =   -0.41841860591874017976596344143339\\
b_5 =   -0.044417392968668277474354150469769 \; \; \; &
b_6 =    0.27805384113645232493158618493986 \\
 \end{array} 
 \]
 
\[ 
\begin{array}{lll} 
\label{eq:B6s5p5m}
a^\varepsilon_{21} =  1.278053841136452  &
a^\varepsilon_{22 } = 0.278053841136452  & \\
a^\varepsilon_{31} =   -0.973638114602592  \; \; \; \; \; &
a^\varepsilon_{32} = -0.072668314073363   &
a^\varepsilon_{33} = 0.278053841136452   \\
a^\varepsilon_{41} =  1.462881746181012   &
a^\varepsilon_{42} = 0.818779639959430  &
a^\varepsilon_{43} = -1.271751126611331   \\
a^\varepsilon_{44} =  0.278053841136452    &  
a^\varepsilon_{51} =    0.079158039988964  &
a^\varepsilon_{52} = 8.098601175558391  \\
a^\varepsilon_{53} =  -0.795757633006629 &
a^\varepsilon_{54} = -4.026427771797058  \; \; \; & 
a^\varepsilon_{55}  = 0.278053841136452    \\     
a^\varepsilon_{61} =  -0.003644498479103  &
a^\varepsilon_{62} = 1.823599336039399   &
a^\varepsilon_{63} =1.194982506135020  \\
a^\varepsilon_{64} = -1.418418578492855  &
a^\varepsilon_{65} =0.011900225406260  &
a^\varepsilon_{66} =0.278053841136452 \\
 \end{array}  \]


\begin{thebibliography}{1}


\bibitem{broyden}
{\sc C.~Broyden}, {\em A class of methods for solving nonlinear simultaneous
  equations}, Mathematics of Computation, 19 (1965), pp.~577--593.

\bibitem{Burnett2}
{\sc B.~Burnett, S.~Gottlieb, and Z.~J. Grant}, {\em Stability analysis and
  performance evaluation of additive mixed-precision {R}unge-{K}utta methods},
  Commun. Appl. Math. Comput., 6 (2024), p.~705–738.

\bibitem{Burnett1}
{\sc B.~Burnett, S.~Gottlieb, Z.~J. Grant, and A.~Heryudono}, {\em Performance
  evaluation of mixed-precision {R}unge-{K}utta methods}, in 2021 IEEE High
  Performance Extreme Computing Conference (HPEC), 2021, pp.~1--6.

\bibitem{CrouzeixBstab}
{\sc M.~Crouzeix}, {\em Sur la b-stabilité des méthodes de {R}unge-{K}utta},
  Numerische Mathematik, 32 (1979), pp.~75--82.

\bibitem{Driscoll2026}
{\sc J. Driscoll, S. Gottlieb, Z.J. Grant, C. Herrera, T.S. Kakumanu,
M. H. Sawicki, and M. Stephens},
{\em Stable corrections for perturbed diagonally implicit Runge-Kutta methods}, 
Submitted (2026). 
\url{https://arxiv.org/abs/2603.24451}.

\bibitem{GithubSmoothPert}
{\sc J. Driscoll, S. Gottlieb, Z.J. Grant, C. Herrera, T.S. Kakumanu}
{\em Smooth Perturbations DIRK methods} (2026)
\url{https://github.com/Mixed-Precision/SmoothPerturbationsDIRKmethods}

\bibitem{Grant2022}
{\sc Z.~J. Grant}, {\em Perturbed  Runge-Kutta methods for mixed precision
  applications}, Journal of Scientific Computing, 92 (2022), pp.~1--20.

\bibitem{hairer1996solving2}
{\sc E.~Hairer and G.~Wanner}, {\em Solving Ordinary Differential Equations II:
  Stiff and Differential-Algebraic Problems}, vol.~14 of Springer Series in
  Computational Mathematics, Springer-Verlag, Berlin, Heidelberg, 2nd~ed.,
  1996.

 \bibitem{CarpenterRK}
 {\sc C.~A. Kennedy and M.~H. Carpenter}, {\em Diagonally implicit  Runge-Kutta
   methods for ordinary differential equations. a review}, NASA Technical
   Report, NASA/TM–2016–219173 (2016).

\bibitem{kurdi}
{\sc  M.A. Kurdi}, 
{\em Stable High Order Methods for Time Discretization of Stiﬀ Differential Equations}, 
Ph.D. Thesis, Applied Mathematics, Univ. of California, Berkeley, Berkeley (1974).

\bibitem{norsett}
{\sc S.~Nørsett}, {\em Semi explicit Runge-Kutta methods}, Math. and Comp.
  Rpt. 6/74 Dept. of Math., Univ. Trondheim,  (1974).

\end{thebibliography}

\end{document}